\documentclass[ reqno]{amsart}
\usepackage{amssymb,amsmath,amsthm}
\usepackage{color}
\usepackage{enumerate}
\usepackage{fullpage}
\usepackage{url}
\usepackage[Symbol]{upgreek}
\usepackage[unicode,pdfusetitle]{hyperref}
\usepackage[T1]{fontenc}
\usepackage{mathtools}

\usepackage{tikz}
\usetikzlibrary{decorations.pathmorphing, patterns,shapes}

\numberwithin{equation}{section}
\newtheorem{lemma}{Lemma}
\numberwithin{lemma}{section}

\newtheorem{corollary}[lemma]{Corollary}
\newtheorem*{thm*}{Theorem}

\newtheorem{proposition}[lemma]{Proposition}
\newtheorem{open problem}[lemma]{Open problem}

\newtheorem*{fact*}{Fact}
\newtheorem{fact}[lemma]{Fact}

\newtheorem*{claim*}{Claim}

\newtheorem{thmA}{Theorem}
\newtheorem{corA}{Corollary}

\theoremstyle{definition}

\newtheorem{remark}[lemma]{Remark}
\newtheorem{final remark}[lemma]{Final remark}
\newtheorem{example}[lemma]{Example}

\newtheorem{question}[lemma]{Question}

\DeclareMathOperator{\acl}{acl}
\DeclareMathOperator{\scl}{scl}
\DeclareMathOperator{\ecl}{ecl}
\DeclareMathOperator{\cl}{c\ell}
\DeclareMathOperator{\dcl}{dcl}
\DeclareMathOperator{\dif}{d}
\DeclareMathOperator{\id}{id}
\DeclareMathOperator{\lex}{lex}
\DeclareMathOperator{\supp}{supp}
\DeclareMathOperator{\trdeg}{trdeg}
\DeclareMathOperator{\Mon}{Mon}
\DeclareMathOperator{\etrdeg}{etrdeg}
\DeclareMathOperator{\rk}{rk}

\newcommand{\N}{\mathbb{N}}
\newcommand{\F}{\mathbb{F}}

\newcommand{\Q}{\mathbb{Q}}
\newcommand{\R}{\mathbb{R}}
\newcommand{\CC}{\mathbb{C}}
\newcommand{\Z}{\mathbb{Z}}

\newcommand{\cC}{\mathcal C}
\newcommand{\cD}{\mathcal D}
\newcommand{\cG}{\mathcal G}

\newcommand{\cK}{\mathcal K}

\newcommand{\cM}{\mathcal M}

\newcommand{\cP}{\mathcal P}

\newcommand{\av}{{\bar a}}
\newcommand{\bv}{{\bar b}}

\newcommand{\dv}{{\bar d}}
\newcommand{\ev}{{\bar e}}
\newcommand{\mv}{{\bar m}}
\newcommand{\rv}{{\bar r}}
\newcommand{\sv}{{\bar s}}
\newcommand{\uv}{{\bar u}}

\newcommand{\w}{{\bar w}}
\newcommand{\x}{{\bar x}}
\newcommand{\y}{{\bar y}}
\newcommand{\Y}{{\bar Y}}

\newcommand{\eev}{\hat{e}}
\newcommand{\zero}{{\bar 0}}

\renewcommand{\preceq}{\preccurlyeq}
\renewcommand{\succeq}{\succcurlyeq}
\renewcommand{\epsilon}{\varepsilon}

\author{Antongiulio Fornasiero}
\author{Elliot Kaplan}
\email{antongiulio.fornasiero@gmail.com}
\email{ekaplan@mpim-bonn.mpg.de}
\title{Hilbert polynomials for finitary matroids}
\address{Dipartimento di Matematica e Informatica ``Ulisse Dini,'' Viale Morgagni, 67/a, 50134 Firenze, Italy}
\address{Max Planck Institute for Mathematics, Vivatsgasse 7, 53111 Bonn, Germany}
\date{\today}
\begin{document}
\maketitle
\begin{abstract}
We consider a tuple $\Phi = (\phi_1,\ldots,\phi_m)$ of commuting maps on a finitary matroid $X$. We show that if $\Phi$ satisfies certain conditions, then for any finite set $A\subseteq X$, the rank of $\{\phi_1^{r_1}\cdots\phi_m^{r_m}(a): a \in A\text{ and }r_1+\cdots+r_m = t\}$ is eventually a polynomial in $t$ (we also give a multivariate version of the polynomial). This allows us to easily recover Khovanskii's theorem on the growth of sumsets, the existence of the classical Hilbert polynomial, and the existence of the Kolchin polynomial. We also prove some new Kolchin polynomial results for differential exponential fields and derivations on o-minimal fields, as well as a new result on the growth of Betti numbers in simplicial complexes.
\end{abstract}
\setcounter{tocdepth}{1}
\tableofcontents

\section*{Introduction}\label{sec:intro}
\noindent
Eventual polynomial growth is a common theme in combinatorics and commutative
algebra. The quintessential example is the Hilbert function, which measures the
$K$-linear dimension of the graded pieces $M_t$ of a finitely generated
$K[x_1,\ldots,x_m]$-module $M = \bigoplus_tM_t$. 
This function is eventually equal to a polynomial in $t$, called the
\emph{Hilbert polynomial}. 
Another example is due to Kolchin, who showed that given a partial differential
field $(F,\delta_1,\ldots,\delta_m)$ of characteristic $0$ and a tuple $\av$ in a differential field extension of~$F$, the transcendence degree
\[
\trdeg_FF\big((\delta_1^{r_1}\cdots \delta_m^{r_m}\av)_{r_1+\cdots + r_m \leq t}\big)
\]
is eventually equal to a polynomial in $t$~\cite{Ko64}. This polynomial, called the \emph{Kolchin polynomial}, is a foundational object in differential algebra. In the area of additive combinatorics, Khovanskii showed that for finite subsets $A$ and $B$ of a commutative semigroup, the size of the sumset $A+tB$ is eventually polynomial in $t$~\cite{Kh92}.

\smallskip
In this paper, we show that the Hilbert polynomial, the Kolchin polynomial, and Khovanskii's polynomial admit a common generalization in terms of finitary matroids. A \emph{matroid} is a combinatorial structure that axiomatizes the notion of independence. While most matroids studied are finite, it is useful for our purposes to allow infinite matroids, while still requiring that any instance of dependence is witnessed by a finite set. There are many equivalent ways to define a finitary matroid, but for the purposes of this paper, we most frequently use closure operators (where the closure of a set consists of all elements which are not independent from the set) and ranks (where the rank of a finite set is the size of a maximal independent subset).

\smallskip
We consider how the rank of a finite set grows as one applies commuting
operators to the set. We show that, under certain assumptions on the operators,
this rank is eventually polynomial in the total number of times the operators
are applied. 
In the case of the Hilbert polynomial, the closure operator is the $K$-linear
span in~$M$, the rank is the $K$-linear dimension, and the operators are scalar
multiplication by the elements~$x_i$. For the Kolchin polynomial, the closure
operator is the algebraic closure over~$F$, the rank is the transcendence degree
over~$F$, and the operators are $\delta_1,\ldots,\delta_m$, and the identity map. 
For Khovanskii's polynomial, the underlying matroid is the semigroup with trivial closure, so the rank coincides with cardinality, and the operators are addition by elements of~$B$.

\smallskip
We are also able to apply our result in several other settings. We recover a result of Maclagan and Rinc\'{o}n on Hilbert polynomials for homogeneous tropical ideals in the semiring of tropical polynomials~\cite{MR18}, we prove that the Betti numbers of a finite subcomplex of a simplicial complex $\cK$ grow polynomially under an action of $\N^m$ on $\cK$ by simplicial endomorphisms, and we establish an analog of the Kolchin polynomial for difference-differential exponential fields and derivations on o-minimal fields. This last application was the original motivation behind this project, and we believe that our analog can serve the same role in the model theory of o-minimal fields with derivations that the classical Kolchin polynomial serves in the model theory of differential fields.

\smallskip
Khovanskii proved his result by constructing an appropriate graded module and using the existence of the Hilbert polynomial. Kolchin's result can also be proven using the Hilbert polynomial, as was shown by Johnson~\cite{Jo69}. Other known examples of eventual polynomial growth, such as Maclagan and Rinc\'{o}n's result on Hilbert polynomials for homogeneous tropical ideals~\cite{MR18}, are also often proved by reduction to the classical Hilbert polynomial. 
Nathanson and Ruzsa later gave a more elementary proof of Khovanskii's result~\cite{NR02} by reducing the problem to showing that the number of elements in an upward-closed subset of $\N^m$ of a given height $t$ is eventually polynomial in $t$ (where the \emph{height} of $(r_1,\ldots,r_m) \in \N^m$ is the sum $r_1+\cdots+r_m$). This approach is more in line with Kolchin's original proof of the existence of the Kolchin polynomial~\cite[Chapter II, Theorem 6]{Ko73}.

\smallskip
In some sense, the proof of our main result also reduces to counting the number of elements in an upward-closed subset of $\N^m$ or, more precisely, to looking at decreasing functions $\N^m\to \N$. However, the framework of finitary matroids---a fundamentally novel aspect of our approach---allows our main theorem to be quickly and readily applied. 
Many proofs of the existence of the classical Hilbert polynomial make use of generating functions, exploiting the relationship between rational generating functions and eventual polynomial growth through results like Lemma~\ref{lem:rationalimpliespoly} below. We also make use of this relationship in our proof. 
 Key to our approach is Proposition~\ref{prop:decreasingimpliesrational}, which describes the generating function $G_f$ associated to a decreasing function $f\colon\N^m\to \N$. This result appears to be new. 

\smallskip
While our proof is self-contained and fairly elementary, we show in Proposition~\ref{prop:whichonesappear} that once we isolate an appropriate decreasing function, our main theorem can be established \`{a} la Khovanskii by constructing a graded module and using the classical Hilbert polynomial. Consequently, we do not obtain any new numerical polynomials in our setting; see Corollary~\ref{cor:whichonesappear}.

\smallskip
Let us state our main results more precisely. Let $(X,\cl)$ be a finitary
matroid and let $\rk$ be the corresponding rank function; see Subsection~\ref{sec:matroid} for definitions.  Let $m \in \N^{>0}$, and let $\Phi\coloneqq (\phi_1,\ldots,\phi_m)$ be a finite tuple of commuting maps $X\to X$. The tuple $\Phi$ is said to be a \textbf{triangular system} if
\[
a \in \cl(B)\Longrightarrow \phi_i a \in \cl(\phi_1B \cup \cdots \cup \phi_iB)
\]
for all $i \in \{1,\ldots,m\}$, all $a \in X$, and all $B \subseteq X$. For $\rv= (r_1,\ldots,r_m) \in \N^m$, we let $|\rv| = r_1+\cdots+r_m$, and we let $\phi^{\rv}\colon X\to X$ be the composite map $\phi^{\rv}=\phi_1^{r_1}\cdots\phi_m^{r_m}$. For $t \in \N$ and $A \subseteq X$, put 
\[
\Phi^{(t)}(A)\coloneqq\big\{\phi^{\rv}(a):a \in A\text{ and }|\rv| = t\big\}.
\]
We prove the following:

\begin{thm*}
Suppose that $\Phi$ is a triangular system, and let $A, B\subseteq X$ with $A$ finite. Then there is a polynomial $P \in \Q[Y]$ of degree at most $m-1$ such that 
\[
\rk(\Phi^{(t)}(A)|\Phi^{(t)}(B))=P(t)
\]
for all sufficiently large $t \in \N$.
\end{thm*}

The above theorem is a special case of a multivariate version, which can be used to recover multivariate generalizations of the results considered above, such as Nathanson's generalization of Khovanskii's sumset theorem~\cite{Na00} and Levin's multivariate generalizations of the Kolchin polynomial~\cite{Le20}. Fix
\[
0=m_1<m_2<\cdots<m_{k}<m_{k+1}=m,
\]
and set $d_i\coloneqq m_{i+1}- m_i$ for $i \in \{1,\ldots, k\}$. For each $i$, we set $\Phi_i\coloneqq(\phi_{m_i+1},\phi_{m_i+2},\ldots,\phi_{m_i+d_i})$, and we call the tuple $(\Phi_1,\ldots,\Phi_k)$ a \textbf{partition of $\Phi$}. For a tuple $\rv = (r_1,\ldots,r_m) \in \N^m$, we set 
\[
\Vert \rv\Vert\coloneqq(r_{m_1+1}+\cdots+r_{m_1+d_1},\ldots,r_{m_k+1}+\cdots+r_{m_k+d_k})\in\N^k,
\]
and for $\sv \in \N^k$ and $A \subseteq X$, we put $\Phi^{(\sv)}(A)\coloneqq \big\{\phi^{\rv}(a):a \in A\text{ and }\Vert\rv\Vert =\sv\big\}$.

\begin{thmA}\label{thm:hilbert}
Suppose that each $\Phi_i$ is a triangular system, and let $A, B\subseteq X$ with $A$ finite. Then there is a polynomial $P^\Phi_{A|B} \in \Q[Y_1,\ldots,Y_k]$ of degree at most $d_i-1$ in each variable $Y_i$ such that 
\[
\rk(\Phi^{(\sv)}(A)|\Phi^{(\sv)}(B))=P^\Phi_{A|B}(\sv)
\]
for $\sv = (s_1,\ldots,s_k) \in \N^k$ with $\min\{s_1,\ldots,s_k\}$ sufficiently large.
\end{thmA}

We call the polynomial $P^\Phi_{A|B}$ in Theorem~\ref{thm:hilbert} the \textbf{dimension polynomial of $A$ over $B$ with respect to the partition $(\Phi_1,\ldots,\Phi_k)$}. To apply Theorem~\ref{thm:hilbert} to the case of differential fields, we need to consider a slightly more general framework than a triangular system. We say that the system $\Phi= (\phi_1,\ldots,\phi_m)$ is \textbf{quasi-triangular} if the augmented system $(\id,\phi_1,\ldots,\phi_m)$ is triangular, where $\id\colon X\to X$ is the identity map. Let $\preceq$ denote the product order on $\N^k$, and for $\sv \in \N^k$ and $A \subseteq X$, put $\Phi^{\preceq(\sv)}(A)\coloneqq \big\{\phi^{\rv}(a):a \in A\text{ and }\Vert\rv\Vert \preceq\sv\big\}$.
\begin{corA}\label{cor:kolchin}
Suppose that each $\Phi_i$ is a quasi-triangular system, and let $A, B\subseteq X$ with $A$
finite. 
Then there is a polynomial $Q^\Phi_{A|B} \in \Q[Y_1,\ldots,Y_k]$ of degree at most $d_i$ in each variable $Y_i$ such that 
\[
\rk(\Phi^{\preceq(\sv)}(A)|\Phi^{\preceq(\sv)}(B))=Q^\Phi_{A|B}(\sv)
\]
for $\sv = (s_1,\ldots,s_k) \in \N^k$ with $\min\{s_1,\ldots,s_k\}$ sufficiently large.
\end{corA}

We call the polynomial $Q^\Phi_{A|B}$ in Corollary~\ref{cor:kolchin} the \textbf{cumulative dimension polynomial of $A$ over $B$ with respect to the partition $(\Phi_1,\ldots,\Phi_k)$}. Any triangular system is quasi-triangular, so Corollary~\ref{cor:kolchin} applies in a strictly broader context than Theorem~\ref{thm:hilbert}. Indeed, there are quasi-triangular systems for which the conclusion of Theorem~\ref{thm:hilbert} doesn't hold; see Subsection~\ref{subsec:counterexample}. As we will see, derivations on fields form a quasi-triangular system with respect to the matroid of algebraic closure.

\smallskip

In addition to the dimension and cumulative dimension polynomials, we define closure operators $\cl^\Phi$ and $\cl^\Phi_*$ on $X$, called the \textbf{$\Phi$-closure} and \textbf{$\Phi_*$-closure}, respectively, with respect to the partition $(\Phi_1,\ldots,\Phi_k)$ as follows:
\begin{align*}
a \in \cl^\Phi(B) &:\Longleftrightarrow \rk(\Phi^{(\sv)}(a)|\Phi^{(\sv)}(B))<|\Phi^{(\sv)}|\text{ for some }\sv \in \N^k\\
a \in \cl^{\Phi_*}(B) &:\Longleftrightarrow \rk(\Phi^{\preceq(\sv)}(a)|\Phi^{\preceq(\sv)}(B))<|\Phi^{\preceq(\sv)}|\text{ for some }\sv \in \N^k.
\end{align*}
Our second theorem relates these closure operators to the dimension and cumulative dimension polynomials.

\begin{thmA}\label{thm:phiclosure}
Let $A, B \subseteq X$ with $A$ finite. If each $\Phi_i$ is a triangular system, then $(X,\cl^\Phi)$ is a finitary matroid. The rank function $\rk^\Phi$ corresponding to this matroid satisfies the identity
\[
\rk^\Phi(A|B)=\lim\limits_{\min(\sv)\to \infty}\frac{\rk(\Phi^{(\sv)}(A)|\Phi^{(\sv)}(B))}{|\Phi^{(\sv)}|}.
\]
We also have
\[
P^\Phi_{A|B}(Y_1,\ldots,Y_k)=\frac{\rk^\Phi(A|B)}{(d_1-1)!\cdots (d_k-1)!}Y_1^{d_1-1}\cdots Y_k^{d_k-1}+\text{lower degree terms}.
\]
Likewise, if each $\Phi_i$ is quasi-triangular, then $(X,\cl^{\Phi_*})$ is a finitary matroid with rank function $\rk^{\Phi_*}$ satisfying the identities
\begin{align*}
\rk^{\Phi_*}(A|B) &= \lim\limits_{\min(\sv)\to \infty}\frac{\rk(\Phi^{\preceq(\sv)}(A)|\Phi^{\preceq(\sv)}(B))}{|\Phi^{\preceq(\sv)}|},\\
Q^\Phi_{A|B}(Y_1,\ldots,Y_k) &= \frac{\rk^{\Phi_*}(A|B)}{d_1!\cdots d_k!}Y_1^{d_1}\cdots Y_k^{d_k}+\text{lower degree terms}.
\end{align*}
\end{thmA}

The $\Phi_*$-closure can be thought of as an analog of differentially algebraic closure; see~\cite[Chapter II, Section 8]{Ko73}. In fact, if each $\phi_i$ is a field derivation, then $\Phi_*$-closure is exactly the differentially algebraic closure, and the corresponding rank $\rk^{\Phi_*}$ is differential transcendence degree.

\subsection*{Outline}
After some preliminaries in Section~\ref{sec:prelim}, we prove our main theorems in Sections~\ref{sec:hilbert} and~\ref{sec:phiclosure}. In Section~\ref{sec:classical}, we collect some classical consequences of Theorem~\ref{thm:hilbert}, and in Section~\ref{sec:betti}, we consider an application to simplicial complexes. Some applications of Corollary~\ref{cor:kolchin} for difference-differential fields (as well as difference-differential exponential fields and differential o-minimal fields) are considered in Section~\ref{sec:dfields}.

\subsection*{Acknowledgements}
The first author was partially supported by GNSAGA--INdAM. The second author was located at the Fields Institute for Research in Mathematical Sciences and was supported by the National Science Foundation under Award No. 2103240. The authors would like to thank Piotr Kowalski, George Shakan, Francesco Gallinaro, James Freitag, and Matthias Aschenbrenner for helpful conversations. We would also like the referees for useful feedback, including an improvement to Proposition~\ref{prop:invariant}.

\section{Preliminaries}\label{sec:prelim}
\subsection{Notation and conventions}
Throughout, $\N$ denotes the set of natural numbers $\{0,1,2,\ldots\}$. Let $d \in \N^{>0}$, and let $\rv = (r_1,\ldots,r_d)$ and $\sv = (s_1,\ldots,s_d)$ range over $\N^d$. We write $\min(\rv)$ to mean $\min\{r_1,\ldots,r_d\}$. Let $\preceq$ denote the product order on $\N^d$, so
\[
\rv \preceq \sv :\Longleftrightarrow r_i \leq s_i \text{ for each }i = 1,\ldots,d,
\]
and let $<_{\lex}$ denote the lexicographic order on $\N^d$ with emphasis on the last coordinate, so 
\[
\rv <_{\lex} \sv :\Longleftrightarrow \text{there is $i \in\{1,\ldots,d\}$ such that $r_i <s_i$ and $r_j = s_j$ for $i<j\leq d$}.
\]
Let $\zero_d$ be the tuple $(0,\ldots,0)$ consisting of $d$ zeros, so $\zero_d$ is the minimal element of $\N^d$ with respect to both of the orders $\preceq$ and $\leq_{\lex}$. For $i \in \{1,\ldots,d\}$, let $\eev_{i,d}\coloneqq (0,\ldots,1,\ldots,0)$ be the tuple which consists of a 1 in the $i$th spot and zeros everywhere else.

\smallskip
We fix, for the remainder of this paper, numbers $0<k\leq m \in \N$, as well as a partition
\[
0=m_1<m_2<\cdots<m_{k}<m_{k+1}=m.
\]
We set $d_i\coloneqq m_{i+1}- m_i$ for $i \in \{1,\ldots, k\}$, and for a tuple $\rv = (r_1,\ldots,r_m) \in \N^m$, we set 
\[
\Vert \rv\Vert\coloneqq(r_{m_1+1}+\cdots+r_{m_1+d_1},\ldots,r_{m_k+1}+\cdots+r_{m_k+d_k})\in\N^k.
\]

\subsection{Finitary matroids, triangular systems, and quasi-triangular systems}
\label{sec:matroid}
Recall that a \textbf{finitary matroid} consists of a set $X$, together with a map $\cl\colon\cP(X)\to \cP(X)$ which satisfies the following conditions:
\begin{enumerate}
\item Reflexivity: $A \subseteq \cl(A)$;
\item Monotonicity: if $A\subseteq B\subseteq X$, then $cl(A) \subseteq \cl(B)$;
\item Idempotence: $\cl(\cl(A))=\cl(A)$ for $A \subseteq X$;
\item Finite character: if $A \subseteq X$ and $a\in \cl(A)$, then $a \in \cl(A_0)$ for some finite subset $A_0\subseteq A$;
\item Steinitz exchange: For $a,b \in X$ and $A \subseteq X$, if $a \in \cl(A\cup \{b\})\setminus\cl(A)$, then $b \in \cl(A\cup\{a\})$.
\end{enumerate}

\smallskip
More on finitary matroids can be found in~\cite{Ox92}, where they are called \emph{independence spaces}. Finitary matroids often appear in model theory, where they are called \emph{pregeometries}; see~\cite[Appendix C.1]{TZ12}. For the remainder of this paper, we fix a finitary matroid $(X,\cl)$, a positive natural number $m \in \N^{>0}$, and a finite tuple $\Phi\coloneqq (\phi_1,\ldots,\phi_m)$ of commuting maps $X\to X$. We will usually use $a,b$ to denote elements of $X$ and $A, B$ to denote subsets of $X$. We will often abuse notation and write things like ``$a\in \cl(ABb)$'' to mean ``$a\in\cl(A\cup B \cup \{b\})$.'' 

\smallskip
For $A \subseteq X$, we let $\cl_A$ denote the following closure operator:
\[
a \in \cl_A(B) :\Longleftrightarrow a \in \cl(AB).
\]
Then $(X,\cl_A)$ is also a finitary matroid, called the \textbf{relativization of $(X,\cl)$ at $A$}. We say that $B$ is \textbf{$\cl$-independent over $A$} if $b\not\in \cl_A(B \setminus \{b\})$ for all $b \in B$. A \textbf{basis for $B$ over $A$} is a subset $B_0\subseteq B$ which is $\cl$-independent over $A$ such that $B \subseteq \cl_A(B_0)$. Steinitz exchange ensures that any two bases for $B$ over $A$ have the same cardinality, called the \textbf{rank of $B$ over $A$} and denoted $\rk(B|A)$. We just write $\rk(B)$ for $\rk(B|\emptyset)$, and we use $\rk_A$ for the rank corresponding to the relativization $(X,\cl_A)$, so $\rk(B|A) = \rk_A(B)$.

\smallskip
Let $\Theta$ be the free commutative monoid on $\Phi$, so $\Theta$ consists of all operators $\phi^{\rv} \coloneqq \phi_1^{r_1}\cdots\phi_m^{r_m}$ for $\rv \in \N^m$. Note that $\phi^{\zero_m}$ is the identity map on $X$ (and also the identity element of $\Theta$) and that $\phi^{\eev_{i,m}} =\phi_i$ for $i = 1,\ldots,m$. For $\theta \in \Theta$, we let $\theta(A) \coloneqq \{\theta a:a \in A\}$, and for $\Theta_0\subseteq \Theta$, we let 
\[
\Theta_0(A)\coloneqq\bigcup_{\theta \in \Theta}\theta(A).
\]
Recall from the introduction that $\Phi$ is a \textbf{triangular system} for $(X,\cl)$ if
\[
a \in \cl(B)\Longrightarrow \phi_i a \in \cl(\phi_1B\cdots\phi_iB)
\]
for every $B \subseteq X$ and for each $i \in\{1,\ldots,m\}$. 
If $\Phi$ is a triangular system for $(X,\cl)$ and $A \subseteq X$ is closed under each map $\phi_i$, then one can easily verify that $\Phi$ is a triangular system for the relativization $(X,\cl_A)$. The following lemma on triangular systems will be used in the proof of the main theorem.

\begin{lemma}\label{lem:tritorank}
$\Phi$ is a triangular system if and only if for any $A,B \subseteq X$ and for each $i \in \{1,\ldots,m\}$, we have
\[
\rk\!\big(\phi_i(A)\big|\phi_1(AB)\cdots \phi_{i-1}(AB)\phi_i(B)\big)\leq\rk(A|B).
\]
\end{lemma}
\begin{proof}
Suppose that $\Phi$ is a triangular system.
Let $A_0$ be a $\cl$-basis for $A$ over $B$, so $A \subseteq \cl(A_0B)$. Since $\Phi$ is a triangular system, we have
\[
\phi_i(A)\subseteq\cl\!\big(\phi_1(A_0B)\cdots\phi_i(A_0B)\big)\subseteq\cl\!\big(\phi_1(AB)\cdots\phi_{i-1}(AB)\phi_i(A_0B)\big).
\]
This gives
\[
\rk\!\big(\phi_i(A)\big|\phi_1(AB)\cdots \phi_{i-1}(AB)\phi_i(B)\big)\leq|\phi_i(A_0)|\leq|A_0|=\rk(A|B).
\]
For the converse, let $a \in \cl(B)$. Then,
\[
\rk\!\big(\phi_{i}(a) \big| \phi_1(aB)\cdots \phi_{i-1}(aB)\phi_i(B)\big)\leq\rk(a|B)=0.
\]
By induction on $i = 1, \ldots, m$, we conclude that $\phi_ia \in\cl( \phi_1B\cdots \phi_{i-1}B\phi_iB)$.
\end{proof}

For the remainder of the paper, we let $(\Phi_1,\ldots,\Phi_k)$ be the partition of $\Phi$ given in the introduction, so $\Phi_i\coloneqq(\phi_{m_i+1},\phi_{m_i+2},\ldots,\phi_{m_i+d_i})$ for each $i$. 
For $\sv \in \N^k$ and $A \subseteq X$, we recall the sets
\[
\Phi^{(\sv)}(A)\coloneqq \big\{\phi^{\rv}(a):a \in A\text{ and }\Vert\rv\Vert =\sv\big\},\qquad \Phi^{\preceq(\sv)}(A)\coloneqq \big\{\phi^{\rv}(a):a \in A\text{ and }\Vert\rv\Vert \preceq\sv\big\}.
\]
For each $i\in \{1,\ldots,k\}$, let 
\[
(\Phi_{i})_*\coloneqq(\id,\phi_{m_i+1},\phi_{m_i+2},\ldots,\phi_{m_i+d_i}),
\]
 where $\id\colon X \to X$ is the identity map. Then $\big((\Phi_1)_*,(\Phi_2)_*,\ldots,(\Phi_k)_*\big)$ is a partition of the augmented system
\[
\Phi_*\coloneqq(\id,\phi_{m_1+1},\ldots,\phi_{m_1+d_1};\ldots;\id,\phi_{m_k+1},\ldots,\phi_{m_k+d_k}).
\]
Given $\sv \in \N^k$ and $\rv \in \N^m$ with $\Vert \rv\Vert \preceq \sv$, the map
$\phi^{\rv} \in \Phi^{\preceq (\sv)}$ acts the same way on $X$ as the map $\id^{\sv-\lVert\rv\rVert}\phi^{\rv} \in \Phi_*^{(\sv)}$, so we may identify $\Phi^{\preceq (\sv)}$ and $\Phi_*^{(\sv)}$. For each $i$, if $\Phi_i$ is quasi-triangular, then $(\Phi_i)_*$ is triangular, so many of our results on quasi-triangular systems will follow from the corresponding result on triangular systems, applied with $\Phi_*$ in place of~$\Phi$.

\smallskip
The main examples of (quasi)-triangular systems studied in this paper are tuples
of (quasi)-endomorphisms. Let $\psi\colon X \to X$ be a map. We say that $\psi$ is an
\textbf{endomorphism of $(X,\cl)$} (a.k.a. ``strong map'' in the matroid literature) if
\[
a \in \cl(B)\Longrightarrow \psi a \in \cl(\psi B).
\]
We say that $\psi$ is a \textbf{quasi-endomorphism of $(X,\cl)$} if 
\[
a \in \cl(B)\Longrightarrow \psi a \in \cl(B\psi B).
\]
If $\Phi$ is a (quasi)-triangular system, then $\phi_1$ is necessarily a (quasi)-endomorphism, and if $\phi_1,\ldots,\phi_m$ are (quasi)-endomorphisms, then $\Phi$ is (quasi)-triangular. Quasi-endomorphisms were first considered in~\cite[Section 3.1]{FK21}.

\subsection{Generating functions}
Let $\Y = (Y_1,\ldots,Y_k)$ be a tuple of variables. Given $\sv = (s_1,\ldots,s_k) \in \N^k$, we write $\Y^\sv$ for the monomial $Y_1^{s_1}\cdots Y_k^{s_k}$. A polynomial $P\in \Q[\Y]$ is said to have \textbf{degree at most $\sv$} if $P$ has degree at most $s_i$ in each variable $Y_i$, that is, if 
\[
P(\Y)=a\Y^{\sv} + \text{lower degree terms}
\]
for some $a \in \Q$.

\smallskip
Let $f\colon\N^k \to \N$ be a function. The \textbf{generating function of $f$} is the multivariate power series 
\[
G_f(\Y)=\sum_{\sv \in \N^k} f(\sv) \Y^\sv\in\Z[\![ \Y]\!].
\]
The following is well-known; we include here essentially the same proof given in~\cite[Lemma 2.1]{KMU20}.

\begin{lemma}\label{lem:rationalimpliespoly}
Suppose that $G_f$ is a rational function with numerator $R(\Y) \in \Q[\Y]$ of degree at most $\mv$ and denominator $(1-Y_1)^{d_1}\cdots (1-Y_k)^{d_k}$. Then there is $P\in \Q[\Y]$ of degree at most $(d_1 - 1,\ldots,d_k-1)$ such that $f(\sv) = P(\sv)$ whenever $\sv \succeq \mv$. This polynomial $P$ has the form
\[
P(\Y) = \frac{R(1,\ldots,1)}{(d_1-1)!\cdots(d_k-1)!}Y_1^{d_1-1}\cdots Y_k^{d_k-1}+\text{lower degree terms}.
\]
\end{lemma}
\begin{proof}
Write $R(\Y) = \sum_{\rv\preceq \mv}a_{\rv}\Y^{\rv}$, so
\[
G_f(\Y)=\frac{R(\Y)}{(1-Y_1)^{d_1}\cdots (1-Y_k)^{d_k}}= \sum_{\rv\preceq \mv}\frac{a_{\rv}\Y^{\rv}}{(1-Y_1)^{d_1}\cdots (1-Y_k)^{d_k}}.
\]
For each $\rv$, a simple computation gives
\[
\frac{a_{\rv}\Y^{\rv}}{(1-Y_1)^{d_1}\cdots (1-Y_k)^{d_k}} = \sum_{\sv\succeq \rv}a_{\rv}\binom{s_1-r_1+d_1-1}{d_1-1}\cdots \binom{s_k-r_k+d_k-1}{d_k-1}\Y^{\sv}
\]
Comparing coefficients, we get for each $\sv \succeq \mv$ that 
\[
f(\sv)= \sum_{\rv\preceq \mv}a_{\rv}\binom{s_1-r_1+d_1-1}{d_1-1}\cdots \binom{s_k-r_k+d_k-1}{d_k-1}.
\]
Putting $P(\Y) \coloneqq \sum_{\rv\preceq \mv}a_{\rv}\binom{Y_1-r_1+d_1-1}{d_1-1}\cdots \binom{Y_k-r_k+d_k-1}{d_k-1}$, we have $f(\sv) = P(\sv)$ for $\sv \succeq \mv$. Note that 
\begin{align*}
P(\Y) &= \frac{\sum_{\rv\preceq \mv}a_{\rv}}{(d_1-1)!\cdots(d_k-1)!}Y_1^{d_1-1}\cdots Y_k^{d_k-1}+\text{lower degree terms}\\
&= \frac{R(1,\ldots,1)}{(d_1-1)!\cdots(d_k-1)!}Y_1^{d_1-1}\cdots Y_k^{d_k-1}+\text{lower degree terms}.\qedhere
\end{align*}
\end{proof}

The function $f$ is said to be \textbf{decreasing} if $f(\rv) \leq f(\sv)$ whenever $\rv \succeq \sv$. Suppose that $f$ is decreasing. For $n \in \N$, set
\[
S_n(f)\coloneqq\{\sv \in \N^k:f(\sv) \leq n\},
\]
so each $S_n(f)$ is a $\preceq$-upward closed subset of $\N^k$ and $S_n(f) = \N^k$ for
$n \geq f(\zero_k)$. Let $M_n(f)$ be the set of $\preceq$-minimal elements of $S_n(f)$,
so each $M_n(f)$ is finite by Dickson's lemma. Set $M(f) \coloneqq \bigcup_{n\in \N}
M_n(f)$, and let $\mv(f)$ be the $\preceq$-least upper bound of $M(f)$ (notice that
the set $M(f)$ is finite).

\begin{proposition}\label{prop:decreasingimpliesrational}
If $f$ is decreasing, then $G_f$ is a rational function with numerator of degree at most $\mv(f)$ and denominator $(1-Y_1)\cdots (1-Y_k)$.
\end{proposition}
\begin{proof}
Suppose that $f$ is decreasing, set $H \coloneqq (1-Y_1)\cdots (1-Y_k)G_f$, and let $\mv(f) = (m_1,\ldots,m_k)$. We need to show for each $i \in \{1,\ldots,k\}$ and each $\sv = (s_1,\ldots,s_k)\in \N^k$ that if $s_i> m_i$, then the coefficient of $\Y^{\sv}$ in $H$ is zero. By symmetry, it suffices to consider the case $i = k$. Take power series $H_0,H_1,H_2,\ldots$ in the variables $(Y_1,\ldots,Y_{k-1})$ such that
\[
H=H_0+ H_1Y_k+H_2Y_k^2 + \cdots.
\]
We fix $t > m_k$, and we will show that $H_t = 0$. Distributing $(1-Y_k)$ through the series $G_f$, we see that
\[
H_t(Y_1,\ldots,Y_{k-1})=(1-Y_1)\cdots (1-Y_{k-1})\sum_{\rv \in \N^{k-1}}(f(\rv,t)- f(\rv,t-1))Y_1^{r_1}\cdots Y_{k-1}^{r_{k-1}},
\]
so it suffices to show that $f(\rv,t)= f(\rv,t-1)$ for each $\rv \in \N^{k-1}$. Let $\rv$ be given and let $n = f(\rv,t)$. Take $\sv \in M_n(f)$ with $\sv \preceq (\rv,t)$. Since $s_k \leq m_k< t$, we see that $(\rv,t-1) \succeq \sv$ as well, so $f(\rv,t-1) \leq n$. Since $f$ is decreasing, we conclude that $f(\rv,t-1) = n= f(\rv,t)$.
\end{proof}

\begin{remark}
We have another proof of Proposition~\ref{prop:decreasingimpliesrational}, using that the partial order on decreasing functions $\N^k\to \N$ given by $f\leq g:\Longleftrightarrow f(\sv)\leq g(\sv)$ for all $\sv \in \N^k$ is well-founded. The proof is as follows:\ let $f\colon \N^k\to \N$ be decreasing and assume that Proposition~\ref{prop:decreasingimpliesrational} holds for all decreasing functions $\N^k\to \N$ less than $f$, as well as all decreasing functions $\N^{k-1}\to \N$ (both base cases hold trivially). Let $g\colon \N^k\to \N$ be the function $(\rv,t) \mapsto f(\rv,t+1)$, so $g \leq f$, and let $h\colon \N^{k-1}\to \N$ be the function $\rv\mapsto f(\rv,0)$. First, consider the case that $g = f$. Then $f(\rv,t) = f(\rv,0) = h(\rv)$ for all $\rv \in \N^{k-1}$ and all $t \in \N$ and so $\mv(f) = (\mv(h),0)$. We have
\[
G_f(\Y)=\sum_{t \in \N}\sum_{\rv \in \N^{k-1}}f(\rv,0)\Y^{(\rv,t)}=(1-Y_k)^{-1}\sum_{\rv \in \N^{k-1}}h(\rv)Y_1^{r_1}\cdots Y_{k-1}^{r_{k-1}}=\frac{G_h(Y_1,\ldots,Y_{k-1})}{(1-Y_k)},
\]
so Proposition~\ref{prop:decreasingimpliesrational} holds for $f$ by our induction hypothesis. Now, consider the case that $g<f$. In this case, $\mv(f)$ is the $\preceq$-least upper bound of $(\mv(h),0)$ and $\mv(g) + \eev_{k,k}$. We have
\[
G_f(\Y)=\sum_{\rv \in \N^{k-1}}f(\rv,0)\Y^{(\rv,0)} + \sum_{t \in \N}\sum_{\rv \in \N^{k-1}}f(\rv,t+1)\Y^{(\rv,t+1)}=G_h(Y_1,\ldots,Y_{k-1}) + Y_kG_g(\Y),
\]
and we again conclude that Proposition~\ref{prop:decreasingimpliesrational} holds for $f$ by our induction hypothesis.
\end{remark}

\section{The dimension and cumulative polynomials}\label{sec:hilbert}
\noindent
In this section, we prove Theorem~\ref{thm:hilbert} and Corollary~\ref{cor:kolchin}. For the remainder of this section, let $A, B \subseteq X$ with $A$ finite.

\subsection{The dimension polynomial}\label{subsec:proof}
In this subsection, we prove Theorem~\ref{thm:hilbert}, and we sketch a slight generalization in Remark~\ref{remark:hilbert2} below. We assume for the remainder of the subsection that each part of the partition $\Phi_i$ is a triangular system. For $\uv \in \N^m$, set 
\[
\Theta_{\uv}\coloneqq\{\phi^{\rv}:\Vert \rv\Vert =\Vert \uv\Vert \text{ and } \rv<_{\lex}\uv \},\qquad f^\Phi_{A|B}(\uv)\coloneqq\rk\!\big(\phi^{\uv}(A)\big|\Theta_{\uv}(A)\Phi^{(\Vert u\Vert)}(B)\big).
\]
Then $f^\Phi_{A|B}$ is bounded above by $|A|$, and $\sum_{\Vert \uv\Vert =\sv}f^\Phi_{A|B}(\uv) =\rk(\Phi^{(\sv)}(A)|\Phi^{(\sv)}(B))$ for each $\sv \in \N^k$.

\begin{lemma}\label{lem:hilbSupward}
The function $f^\Phi_{A|B}$ is decreasing.
\end{lemma}
\begin{proof}
Let $\uv' \succeq \uv$ be given. We may assume that 
\[
\uv'=\uv + \eev_{m_i+d,m}
\]
for some $i \in \{1,\ldots,k\}$ and some $d \in \{1,\ldots,d_i\}$. Since $\Phi_i = (\phi_{m_i+ 1},\ldots,\phi_{m_i+d_i})$ is a triangular system, Lemma~\ref{lem:tritorank} tells us that 
\[
\rk\!\big(\phi_{m_i+d}\phi^{\uv}(A)\big|\{\phi_{m_i+j}\phi^{\uv}(A):0<j<d\} \cup\{\phi_{m_i+j}(\Theta_{\uv}(A)\Phi^{(\Vert\uv\Vert)}(B)):0<j\leq d\}\big)\leq f^\Phi_{A|B}(\uv).
\]
For $j\in \{1,\ldots,d-1\}$, we have $\Vert\eev_{m_i+j,m}\Vert =\Vert \eev_{m_i+d,m}\Vert$ and $\eev_{m_i+j,m}<_{\lex} \eev_{m_i+d,m}$, so 
\[
\Vert\eev_{m_i+j,m}+\uv\Vert=\Vert\eev_{m_i+d,m}+\uv\Vert=\Vert \uv'\Vert,\qquad \eev_{m_i+j,m}+\uv <_{\lex} \eev_{m_i+d,m}+\uv=\uv'.
\]
It follows that $\phi_{m_i+j}\phi^{\uv} \in \Theta_{\uv'}$ for $j\in \{1,\ldots,d-1\}$. Likewise, for $\phi^{\rv} \in \Theta_{\uv}$ and $j\in \{1,\ldots,d\}$, we have 
\[
\Vert\eev_{m_i+j,m}+\rv\Vert=\Vert \eev_{m_i+j,m}+\uv\Vert =\Vert \uv'\Vert,\qquad
\eev_{m_i+j,m}+\rv <_{\lex} \eev_{m_i+j,m}+\uv \leq_{\lex} \uv',
\]
so $\phi_{m_i+j}\Theta_{\uv}\subseteq \Theta_{\uv'}$. Finally, we have $\phi_{m_i+j}(\Phi^{(\Vert\uv\Vert)}) \subseteq\Phi^{(\Vert\uv+\eev_{m_i+j,m}\Vert)} = \Phi^{(\Vert\uv'\Vert)}$ for $j\in \{1,\ldots,d\}$, so
\[
f^\Phi_{A|B}(\uv')\leq\rk\!\big(\phi_{m_i+d}\phi^{\uv}(A)\big|\{\phi_{m_i+j}\phi^{\uv}(A):0<j<d\} \cup\{\phi_{m_i+j}(\Theta_{\uv}(A)\Phi^{(\Vert\uv\Vert)}(B)):0<j\leq d\}\big),
\]
as desired.
\end{proof}

\begin{proof}[\bf Proof of Theorem~\ref{thm:hilbert}]
We let $G^\Phi_{A|B}$ denote the generating function of the function $\sv \mapsto \rk(\Phi^{(\sv)}(A)|\Phi^{(\sv)}(B))$. We have
\[
G^\Phi_{A|B}(\Y)=\sum_{\sv \in \N^k}\rk(\Phi^{(\sv)}(A)|\Phi^{(\sv)}(B))\Y^\sv=\sum_{\sv \in \N^k}\sum_{\Vert \uv\Vert =\sv}f^\Phi_{A|B}(\uv)\Y^\sv=\sum_{\uv \in \N^m}f^\Phi_{A|B}(\uv)\Y^{\Vert \uv\Vert}.
\]
The rightmost sum above is just the generating function of $f^\Phi_{A|B}$ with $Y_1$ substituted for the first $d_1$ variables, $Y_2$ substituted for the next $d_2$ variables, and so on. By Proposition~\ref{prop:decreasingimpliesrational} and Lemma~\ref{lem:hilbSupward}, we have that $G^\Phi_{A|B}$ is a rational function with numerator of degree at most $\Vert\mv(f^\Phi_{A|B})\Vert$ and denominator $(1-Y_1)^{d_1}\cdots (1-Y_k)^{d_k}$. By Lemma~\ref{lem:rationalimpliespoly}, we conclude that there is a polynomial $P^\Phi_{A|B}(\Y) \in \Q[\Y]$ of degree at most $(d_1-1,\ldots,d_k-1)$ such that 
\[
\rk(\Phi^{(\sv)}(A)|\Phi^{(\sv)}(B))=P^\Phi_{A|B}(\sv)
\]
for $\sv \in \N^k$ with $\sv \succeq \Vert\mv(f^\Phi_{A|B})\Vert$.
\end{proof}

\begin{remark}\label{remark:hilbert2}
Let $\Psi = (\psi_1,\ldots,\psi_n)$ be another tuple of commuting maps $X \to X$, and let $(\Psi_1,\ldots,\Psi_k)$ be a partition of $\Psi$. Suppose that for each $i \in \{1,\ldots,k\}$, the tuple $\Phi_i$ is a subtuple of $\Psi_i$. One can prove the following generalization of Theorem~\ref{thm:hilbert}:

\smallskip\noindent
\emph{There is a polynomial $P \in \Q[\Y]$ of degree at most $d_i-1$ in each variable $Y_i$ such that 
\[
\rk(\Phi^{(\sv)}(A)|\Psi^{(\sv)}(B))=P(\sv)
\]
for $\sv \in \N^k$ with $\min(\sv)$ sufficiently large.}

\smallskip\noindent
Note that no assumptions beyond commutativity are imposed on the maps in $\Psi_i \setminus \Phi_i$. All that is required is that each tuple $\Phi_i$ is triangular. 
To prove this generalization, replace the function $f^\Phi_{A|B}$ in the above proof with a function $f^{\Phi,\Psi}_{A|B}$, where
\[
f^{\Phi,\Psi}_{A|B}(\uv)\coloneqq\rk\!\big(\phi^{\uv}(A)\big|\Theta_{\uv}(A)\Psi^{(\Vert u\Vert)}(B)\big)
\]
for $\uv \in \N^m$. With the obvious changes, the proof of Lemma~\ref{lem:hilbSupward} shows that $f^{\Phi,\Psi}_{A|B}$ is decreasing.
\end{remark}

\subsection{The cumulative dimension polynomial}
In this subsection, we prove Corollary~\ref{cor:kolchin}. We also investigate the dominant terms of the cumulative dimension polynomial. We assume for the remainder of the subsection that each part of the partition $\Phi_i$ is a quasi-triangular system.

\begin{proof}[\bf Proof of Corollary~\ref{cor:kolchin}]
We recall the augmented system $\Phi_*=\big((\Phi_1)_*,(\Phi_2)_*,\ldots,(\Phi_k)_*\big)$, where $(\Phi_i)_*=(\id,\phi_{m_i+1},\phi_{m_i+2},\ldots,\phi_{m_i+d_i})$ for each $i\in \{1,\ldots,k\}$. Applying Theorem~\ref{thm:hilbert} with $\Phi_*$ in place of $\Phi$, we get a polynomial $P_{A|B}^{\Phi_*}(\Y)\in \Q[\Y]$ of degree at most $d_i$ in each variable $Y_i$ such that
\[
\rk(\Phi^{\preceq(\sv)}(A)|\Phi^{\preceq(\sv)}(B))=\rk(\Phi_*^{(\sv)}(A)|\Phi_*^{(\sv)}(B))=P^{\Phi_*}_{A|B}(\sv)
\]
for $\sv \in \N^k$ with $\min(\sv)$ sufficiently large. Take $Q^\Phi_{A|B} \coloneqq P_{A|B}^{\Phi_*}$.
\end{proof}

Now we turn to the dominant terms. Given a polynomial 
\[
Q(\Y)=\sum_{\ev\in \N^k} a_{\ev}\Y^\ev \in \Q[\Y],
\]
the \textbf{support of $Q$} is the set $\supp(Q) \coloneqq \{\ev:a_{\ev} \neq 0\} \subseteq \N^k$. We extend the partial order $\preceq$ on $\N^k$ to all of $\R^k$, so for $\x,\y \in \R^k$, we have $\x\preceq \y\Longleftrightarrow x_i\leq y_i$ for $i = 1,\ldots,k$. We set
\[
\cM(Q) \coloneqq \big\{\y \in \R^k:\y\preceq \ev\text{ for some }\ev \in \supp(Q)\big\}.
\]
That is, $\cM(Q)$ is the Minkowski sum of the orthant $(-\infty,0]^k$ and the Newton polytope of $Q$ (the convex hull of the support of $Q$). We define the \textbf{dominant terms of $Q$} to be the terms $a_{\dv}\Y^\dv$ such that $\dv$ is a vertex of $\cM(Q)$.

As is exposited in~\cite[Section 3]{HS14}, the dominant terms of $Q$ can be recovered using limits:\  the term $a_{\dv}\Y^\dv$ is a dominant term of $Q$ if and only if
\begin{equation}\label{eq:dominantterms}
\lim\limits_{t\to \infty}\frac{Q(t^{w_1},t^{w_2},\ldots,t^{w_k})}{a_{\dv}t^{\dv \cdot \w}}=1.
\end{equation}
for some positive vector $\w \in \R^k$ with positive $\Q$-linearly independent entries. Indeed, for such a vector $\w$, the map $\ev \mapsto \ev \cdot \w\colon \N^k\to \R$ is injective, and equation (\ref{eq:dominantterms}) can be readily verified for $\dv \in \supp(Q)$ with $\dv\cdot \w = \max\{\ev\cdot \w:\ev \in \supp(Q)\}$.  It remains to note that $\dv \in \supp(Q)$ is a vertex of $\cM(Q)$ if and only if the sector 
\[
\big\{\w \in (\R^{>0})^k: (\dv-\ev)\cdot\w>0 \text{ for all }\ev \in \supp(Q)\setminus\{\dv\}\big\}
\]
is nonempty. The vectors $\w$ in this sector are exactly those vectors for which (\ref{eq:dominantterms}) holds.

In the case that $B = \Theta(B)$, the dominant terms of $Q^\Phi_{A|B}$ only depend on $\cl(\Theta(A)B)$:

\begin{proposition}\label{prop:invariant}
Let $A,A'$ be finite subsets of $X$, let $B \subseteq X$ with $\Theta(B) = B$, and suppose that $\cl(\Theta(A)B) =\cl(\Theta(A')B)$. Then $Q^\Phi_{A|B}$ and $Q^\Phi_{A'|B}$ have the same dominant terms.
\end{proposition}
\begin{proof}
Take $\sv_0$ with 
\[
A'\subseteq\cl(\Phi^{\preceq (\sv_0)}(A)B),\qquad A\subseteq\cl(\Phi^{\preceq(\sv_0)}(A')B).
\]
Since each $\Phi_i$ is quasi-triangular, a routine induction on $|\sv| = s_1+\cdots+s_k$ gives
\[
\Phi^{\preceq(\sv)}(A')\subseteq\cl(\Phi^{\preceq (\sv_0+\sv)}(A)B),\qquad\Phi^{\preceq(\sv)}(A)\subseteq\cl(\Phi^{\preceq (\sv_0+\sv)}(A')B),
\]
for all $\sv \in \N^k$. It follows that 
\[
\rk(\Phi^{\preceq(\sv)}(A')|B)\leq\rk(\Phi^{\preceq (\sv_0+\sv)}(A)|B),\qquad \rk(\Phi^{\preceq(\sv)}(A)|B)\leq\rk(\Phi^{\preceq (\sv_0+\sv)}(A')|B)
\]
for each $\sv$. Taking $\min(\sv)$ to be sufficiently large, we get 
\[
Q^\Phi_{A'|B}(\sv)\leq Q^\Phi_{A|B}(\sv_0+\sv),\qquad Q^\Phi_{A|B}(\sv)\leq Q^\Phi_{A'|B}(\sv_0+\sv).
\]
The proposition follows easily, using that the limits (\ref{eq:dominantterms}) agree for $Q^\Phi_{A|B}$ and $Q^\Phi_{A'|B}$.
\end{proof}

\section{The $\Phi$-closure and $\Phi_*$-closure operators}\label{sec:phiclosure}
\noindent
As in the previous section, we fix subsets $A, B \subseteq X$ with $A$ finite.
Recall that the \textbf{$\Phi$-closure} operator is given by
\[
a \in \cl^\Phi(B) :\Longleftrightarrow \rk(\Phi^{(\sv)}(a)|\Phi^{(\sv)}(B))<|\Phi^{(\sv)}|\text{ for some }\sv \in \N^k.
\]
The \textbf{$\Phi_*$-closure operator} is defined identically, but with $\Phi_*$ in place of $\Phi$, so 
\[
a \in \cl^{\Phi_*}(B) \Longleftrightarrow \rk(\Phi^{\preceq(\sv)}(a)|\Phi^{\preceq(\sv)}(B))<|\Phi^{\preceq(\sv)}|\text{ for some }\sv \in \N^k.
\]
In this section, we prove Theorem~\ref{thm:phiclosure}. We also examine whether the rank associated to these operators depends on our choice of partition, and we discuss an extension of the $\Phi_*$-closure operator to more general monoid actions by matroid endomorphisms. In order to prove Theorem~\ref{thm:phiclosure}, we need the following lemma.

\begin{lemma}\label{lem:cldegree}
Suppose that each $\Phi_i$ is a triangular system. Then for $a \in X$, we have
\[
\lim\limits_{\min(\sv)\to \infty}\frac{\rk(\Phi^{(\sv)}(a)|\Phi^{(\sv)}(B))}{|\Phi^{(\sv)}|}=\left\{
\begin{array}{ll}
0 & \mbox{if $a \in \cl^\Phi(B)$}\\
1 & \mbox{if $a \not\in \cl^\Phi(B)$}.
\end{array}\right.
\]
\end{lemma}
\begin{proof}
If $a \not\in \cl^\Phi(B)$, then $\rk(\Phi^{(\sv)}(a)|\Phi^{(\sv)}(B))=|\Phi^{(\sv)}|$ for all $\sv \in \N^k$, so 
\[
\lim\limits_{\min(\sv)\to \infty}\frac{\rk(\Phi^{(\sv)}(a)|\Phi^{(\sv)}(B))}{|\Phi^{(\sv)}|}=1.
\]
Suppose $a$ belongs to $\cl^\Phi(B)$, as witnessed by $\sv_0 \in \N^k$. Then we have
\[
\sum_{\Vert\uv\Vert = \sv_0}f^\Phi_{a|B}(\uv)=\rk(\Phi^{(\sv_0)}(a)|\Phi^{(\sv_0)}(B))<|\Phi^{(\sv_0)}|.
\]
It follows that $f^\Phi_{a|B}(\uv_0) = 0$ for some $\uv_0 \in \N^m$ with $\Vert \uv_0 \Vert = \sv_0$. By Lemma~\ref{lem:hilbSupward}, we have $f^\Phi_{a|B}(\uv) = 0$ whenever $\uv \succeq \uv_0$. Let $\sv \in\N^k$ with $\sv \succeq \sv_0$. Then
\[
\rk(\Phi^{(\sv)}(a)|\Phi^{(\sv)}(B))=\sum_{\Vert \uv\Vert = \sv}f^\Phi_{a|B}(\uv)=\sum_{\substack{\Vert \uv\Vert = \sv\\ \uv \not\succeq\uv_0}}f^\Phi_{a|B}(\uv)\leq|\Phi^{(\sv)}|-|\Phi^{(\sv-\sv_0)}|.
\]
Since $\frac{|\Phi^{(\sv-\sv_0)}|}{|\Phi^{(\sv)}|}$ approaches $1$ as $\min(\sv)$ grows, we have
\[
\lim\limits_{\min(\sv)\to \infty}\frac{\rk(\Phi^{(\sv)}(a)|\Phi^{(\sv)}(B))}{|\Phi^{(\sv)}|}=0.\qedhere
\]
\end{proof}

\begin{proof}[\bf Proof of Theorem~\ref{thm:phiclosure}]
Let us first assume that each $\Phi_i$ is a triangular system. We need to show that $(X,\cl^\Phi)$ is a finitary matroid. Monotonicity and finite character are both clear. For idempotence, let $a \in \cl^\Phi(\cl^\Phi(B))$ and, using finite character, take elements $b_1,\ldots,b_n \in \cl^\Phi(B)$ with $a \in \cl^\Phi(b_1,\ldots,b_n)$. We have
\begin{align*}
\rk(\Phi^{(\sv)}(a)|\Phi^{(\sv)}(B)) &\leq \rk(\Phi^{(\sv)}(a)|\Phi^{(\sv)}(b_1,\ldots,b_n)) + \rk(\Phi^{(\sv)}(b_1,\ldots,b_n)|\Phi^{(\sv)}(B))\\
&\leq \rk(\Phi^{(\sv)}(a)|\Phi^{(\sv)}(b_1,\ldots,b_n)) + \sum_{i=1}^n\rk(\Phi^{(\sv)}(b_i)|\Phi^{(\sv)}(B)).
\end{align*}
It follows that
\[
\lim\limits_{\min(\sv)\to \infty}\frac{\rk(\Phi^{(\sv)}(a)|\Phi^{(\sv)}(B))}{|\Phi^{(\sv)}|}\leq\lim\limits_{\min(\sv)\to \infty}\frac{\rk(\Phi^{(\sv)}(a)|\Phi^{(\sv)}(b_1,\ldots,b_n)) + \sum_{i=1}^n\rk(\Phi^{(\sv)}(b_i)|\Phi^{(\sv)}(B))}{|\Phi^{(\sv)}|}.
\]
We conclude by Lemma~\ref{lem:cldegree} that $a \in \cl^\Phi(B)$. Finally, for exchange, suppose that $a\in \cl^\Phi(Bb)\setminus \cl^\Phi(B)$. Since $\cl^\Phi$ has finite character, we may assume that $B$ is finite. Idempotence tells us that $b \not\in \cl^\Phi(B)$, so 
\[
\rk(\Phi^{(\sv)}(a)|\Phi^{(\sv)}(B))=|\Phi^{(\sv)}|=\rk(\Phi^{(\sv)}(b)|\Phi^{(\sv)}(B))
\]
for all $\sv$. It follows that 
\begin{align*}
\rk(\Phi^{(\sv)}(b)|\Phi^{(\sv)}(Ba)) &= \rk(\Phi^{(\sv)}(Bab))-\rk(\Phi^{(\sv)}(Ba))\\
&=\rk(\Phi^{(\sv)}(Bab))-\rk(\Phi^{(\sv)}(a)|\Phi^{(\sv)}(B))-\rk(\Phi^{(\sv)}(B))\\
&= \rk(\Phi^{(\sv)}(Bab))-\rk(\Phi^{(\sv)}(b)|\Phi^{(\sv)}(B))-\rk(\Phi^{(\sv)}(B))=\rk(\Phi^{(\sv)}(a)|\Phi^{(\sv)}(Bb)).
\end{align*}
Since $a \in \cl^\Phi(Bb)$, we see that $b \in \cl^\Phi(Ba)$.

Now we turn to the properties of the rank function $\rk^\Phi$ associated to the matroid $(X,\cl^\Phi)$. Let us show that 
\begin{equation}\label{eq:limitrank}
\lim\limits_{\min(\sv)\to \infty}\frac{\rk(\Phi^{(\sv)}(A)|\Phi^{(\sv)}(B))}{|\Phi^{(\sv)}|}=\rk^\Phi(A|B).
\end{equation}
We prove this by induction on $|A|$, with the case $A = \emptyset$ holding trivially. Suppose that this holds for a given $A$, and let $a \in X \setminus A$. We have
\[
\lim\limits_{\min(\sv)\to \infty}\frac{\rk(\Phi^{(\sv)}(Aa)|\Phi^{(\sv)}(B))}{|\Phi^{(\sv)}|}=\lim\limits_{\min(\sv)\to \infty}\frac{\rk(\Phi^{(\sv)}(A)|\Phi^{(\sv)}(B))+\rk(\Phi^{(\sv)}(a)|\Phi^{(\sv)}(AB))}{|\Phi^{(\sv)}|},
\]
and our induction hypothesis and Lemma~\ref{lem:cldegree} gives
\[
\lim\limits_{\min(\sv)\to \infty}\frac{\rk(\Phi^{(\sv)}(A)|\Phi^{(\sv)}(B))+\rk(\Phi^{(\sv)}(a)|\Phi^{(\sv)}(AB))}{|\Phi^{(\sv)}|}=\rk^\Phi(A|B)+\rk^\Phi(a|AB)=\rk^\Phi(Aa|B).
\]
Finally, we will show that
\[
P^\Phi_{A|B}(\Y)=\frac{\rk^\Phi(A|B)}{(d_1-1)!\cdots (d_k-1)!}Y_1^{d_1-1}\cdots Y_k^{d_k-1}+\text{lower degree terms}.
\]
As $\min(\sv)$ grows, we have
\[
P^\Phi_{A|B}(\sv)=\rk(\Phi^{(\sv)}(A)|\Phi^{(\sv)}(B)),\qquad |\Phi^{(\sv)}|=\frac{s_1^{d_1-1}\cdots s_k^{d_k-1}}{(d_1-1)!\cdots (d_k-1)!}+ o(s_1^{d_1-1}\cdots s_k^{d_k-1}),
\]
so the leading coefficient of $P^\Phi_{A|B}$ is equal to the limit
\[
\lim\limits_{\min(\sv)\to \infty}\frac{P^\Phi_{A|B}(\sv)}{s_1^{d_1-1}\cdots s_k^{d_k-1}}=\lim\limits_{\min(\sv)\to \infty}\frac{\rk(\Phi^{(\sv)}(A)|\Phi^{(\sv)}(B))}{(d_1-1)!\cdots (d_k-1)!|\Phi^{(\sv)}|}=\frac{\rk^\Phi(A|B)}{(d_1-1)!\cdots (d_k-1)!},
\]
where the last equality follows from (\ref{eq:limitrank}).

The second part of Theorem~\ref{thm:phiclosure}, involving quasi-triangular systems and the $\Phi_*$-closure, follows from the first part with $\Phi_*$ in place of $\Phi$.
\end{proof}

\begin{remark}
By Theorem~\ref{thm:phiclosure} and Lemma~\ref{lem:rationalimpliespoly}, the $\Phi$-rank $\rk^\Phi(A|B)$ coincides with the numerator of $G^\Phi_{A|B}(\Y)$, evaluated at the tuple $(1,\ldots,1)$.
\end{remark}

\subsection{Dependence on our choice of partition}
The $\Phi$-closure operator is dependent on our partition $(\Phi_1,\ldots,\Phi_k)$, as the following example illustrates:

\begin{example}
Let $(X,\cl)$ be the set $\Z$ with the trivial closure operator, so the rank of any subset of $\Z$ is its cardinality. Let $a \in \Z$, let $\phi_1$ be the map $x\mapsto x+1$, let $\phi_2 = \phi_1$, and let $\Phi = (\phi_1,\phi_2)$. First, suppose $\Phi$ is partitioned trivially (so $k = 1$). Then $\Phi^{(t)}(a) =\{a+t\}$ for any $t \in \N$, so $\rk(\Phi^{(t)}(a))=1<|\Phi^{(t)}|$ for $t>0$ and $\rk^\Phi(a) = 0$. Now, suppose $\Phi$ is given the partition $(\Phi_1,\Phi_2)$, where $\Phi_1 = (\phi_1)$ and $\Phi_2 = (\phi_2)$. Then $\Phi^{(s_1,s_2)}(a)= \{a+s_1+s_2\}$ for $s_1,s_2 \in \N$, so $\rk(\Phi^{(s_1,s_2)}(a))=1=|\Phi^{(s_1,s_2)}|$ and $\rk^\Phi(a) = 1$.
\end{example}

On the other hand, the $\Phi_*$-closure operator is more robust, as it does not depend on our partition:

\begin{proposition}\label{prop:indeppart}
For $a \in X$, we have $a\in \cl^{\Phi_*}(B)$ if and only if $(\theta a)_{\theta \in \Theta}$ is not $\cl$-independent over $\Theta(B)$.
\end{proposition}
\begin{proof}
Clearly, if $a \in \cl^{\Phi_*}(B)$, then $(\theta a)_{\theta \in \Theta}$ is not $\cl$-independent over $\Theta(B)$. Suppose that $(\theta a)_{\theta \in \Theta}$ is not $\cl$-independent over $\Theta(B)$, and take a finite subset $\Theta_0 \subseteq \Theta$ with $\rk(\Theta_0(a)|\Theta(B))<|\Theta_0|$. Since $\cl$ has finite character, we can take $\sv \in \N^k$ with $\Theta_0 \subseteq \Phi^{\preceq(\sv)}$ such that $\rk(\Theta_0(a)| \Phi^{\preceq(\sv)}(B))<|\Theta_0|$. Set $\Theta_1\coloneqq \Phi^{\preceq(\sv)} \setminus \Theta_0$, so
\[
\rk(\Phi^{\preceq(\sv)}(a)|\Phi^{\preceq(\sv)}(B))\leq\rk(\Theta_0(a)|\Phi^{\preceq(\sv)}(B))+\rk(\Theta_1(a)|\Phi^{\preceq(\sv)}(B))<|\Theta_0|+|\Theta_1|=|\Phi^{\preceq(\sv)}|.
\]
We conclude that $a \in \cl^{\Phi_*}(B)$.
\end{proof}

\begin{corollary}\label{cor:kolchincoefficient}
Suppose that each $\Phi_i$ is a quasi-triangular system. Then $\rk^{\Phi_*}(A|B)$ is the maximal size of a subset $A_0 \subseteq \cl(\Theta(AB))$ such that $(\theta a)_{\theta \in \Theta,a \in A_0}$ is $\cl$-independent over $\Theta(B)$.
\end{corollary}
\begin{proof}
Proposition~\ref{prop:indeppart} tells us that $\rk^{\Phi_*}(A|B)$ is the maximal size of a subset $A_0 \subseteq A$ such that $(\theta a)_{\theta \in \Theta,a \in A_0}$ is $\cl$-independent over $\Theta(B)$. Let $A'$ be a finite subset of $\cl(\Theta(AB))$ such that $(\theta a)_{\theta \in \Theta,a \in A'}$ is $\cl$-independent over $\Theta(B)$. We will show that $|A'|\leq \rk^{\Phi_*}(A|B)$. Again using Proposition~\ref{prop:indeppart}, we may assume that $B = \Theta(B)$. Arguing as in Proposition~\ref{prop:invariant}, we find $\sv_0 \in \N^k$ such that $Q^\Phi_{A'|B}(\sv)\leq Q^\Phi_{A|B}(\sv_0+\sv)$ for all $\sv$ with $\min(\sv)$ sufficiently large. Then $\rk^{\Phi_*}(A'|B)\leq \rk^{\Phi_*}(A|B)$ by Theorem~\ref{thm:phiclosure}, and it remains to note that $|A'| =\rk^{\Phi_*}(A'|B)$.
\end{proof}

\subsection{Monoid actions by endomorphisms}
Suppose that each $\phi_i$ is an endomorphism, that is, suppose
\[
a \in \cl(B)\Longrightarrow \phi_i a \in \cl(\phi_i B)
\]
for each $i$. Then $\phi^{\rv}$ is an endomorphism for each $\rv \in \N^m$, so $(\rv,a)\mapsto \phi^{\rv}(a)$ gives us a monoid action $\N^m\!\curvearrowright\!(X,\cl)$ by endomorphisms. In this case, each $\Phi_i$ is a triangular system, so $(X,\cl^{\Phi_*})$ is also a finitary matroid and 
\[
\rk^{\Phi_*}(A)=\lim\limits_{\min(\sv)\to \infty}\frac{\rk(\Phi^{\preceq(\sv)}(A))}{|\Phi^{\preceq(\sv)}|}
\]
by Theorem~\ref{thm:phiclosure}. Thus, we can view $\rk^{\Phi_*}(A)$ as an average of $\rk(A)$ over the action $\N^m\!\curvearrowright\!(X,\cl)$, where this averaging is taken with respect to the net $\big(\{\rv \in \N^m:\Vert \rv\Vert \preceq \sv\}\big)_{\sv \in \N^k}$ of subsets of $\N^m$.

\smallskip
While our result on polynomial growth seems limited to the context of $\N^m$ acting on $(X,\cl)$, one can make sense of this ``averaging'' for the right action of any cancellative left-amenable monoid, using the main theorem from~\cite{CSCK14}. Let $M$ be a cancellative left-amenable semigroup and let $\alpha\colon X\times M\to X$ be a right action $(X,\cl)\!\curvearrowleft\!M$ by endomorphisms. Given finite subsets $A\subseteq X$ and $S\subseteq M$, we put $r_A(S)\coloneqq \rk(\alpha(A,S))$, where $\alpha(A,S) = \bigcup_{(a,s) \in A\times S} \alpha(a,s)$. It is routine to check that $r_A$, as a map from finite subsets of $M$ to $\N$, satisfies the three conditions in the statement of~\cite[Theorem 1.1]{CSCK14}; for the second condition, one needs to use that $s$ is a matroid endomorphism and that our action is on the \emph{right}. It follows that there is a number $\rk^M(A)$, depending only on $M$ and $A$, such that
\[
\lim_i\frac{\rk(\alpha(A,F_i))}{|F_i|}=\lim_i\frac{r_A(F_i)}{|F_i|}=\rk^M(A)
\]
for any left-F\o lner net $(F_i)_{i \in I}$ of $M$ (such nets always exist). Then $\rk^M(A)$ serves as this ``averaged'' rank.

\section{Applications I:\ Some classical results}\label{sec:classical}
\noindent
In this section, we give a handful of applications in the case that each $\phi_i$ is an endomorphism of $(X,\cl)$.

\subsection{Polynomial growth of sumsets}
Let $G$ be a commutative semigroup. In~\cite{Kh92}, Khovanskii showed that for finite subsets $A, B\subseteq G$, the size of the sumset $A+tB$ is given by a polynomial in $t$ for $t$ sufficiently large (here, $A+ tB$ is the set of all elements $a+b_1+\cdots+b_t$, where $a \in A$ and each $b_i \in B$). Moreover, the degree of this polynomial is less than the size of $B$. Using Theorem~\ref{thm:hilbert}, we can recover a generalization of Khovanskii's theorem, originally proven by Nathanson~\cite{Na00}.

\begin{corollary}[Nathanson]\label{cor:Nathanson}
Let $A,B_1,\ldots,B_k$ be finite subsets of $G$. Then there is a polynomial $P \in \Q[\Y]$ such that
\[
|A+s_1B_1+\cdots+s_kB_k|=P(\sv)
\]
whenever $\min(\sv)$ is sufficiently large. Moreover, the degree of $P$ in each variable $Y_i$ is less than $|B_i|$.
\end{corollary}
\begin{proof}
Let $(X,\cl)$ be the underlying set of the semigroup $G$ with the trivial closure operator, so the rank of any subset of $G$ is its cardinality. Let $b_1,\ldots,b_{m_1}$ be an enumeration of $B_1$, let $b_{m_1+1},\ldots,b_{m_2}$ be an enumeration of $B_2$, and so on. For each $i \in \{1,\ldots,m\}$, let $\phi_i\colon G \to G$ be the map $x \mapsto x+b_i$, so each $\phi_i$ is an endomorphism. Then
\[
\Phi^{(\sv)}(A)=A+s_1B_1+\cdots+s_kB_k
\]
for each $\sv= (s_1,\ldots,s_k) \in \N^k$, and it remains to invoke Theorem~\ref{thm:hilbert}.
\end{proof}

The set $B$ in Theorem~\ref{thm:hilbert} makes no appearance in the argument above. This suggests a slight improvement:

\begin{corollary}\label{cor:avoidB}
Let $A, B_1,\ldots, B_k$ be finite subsets of $G$ and let $B$ be an arbitrary subset of $G$. Then there is a polynomial $P \in \Q[\Y]$ such that
\[
|(A+s_1B_1+\cdots+s_kB_k)\setminus (B+s_1B_1+\cdots+s_kB_k)|=P(\sv)
\]
whenever $\min(\sv)$ is sufficiently large.
\end{corollary}

\subsection{Counting elements in an ideal}
In this next application, we make use of the fact that if $\phi$ is an endomorphism of $(X,\cl)$ and $\phi(C) \subseteq C$ for some $C \subseteq X$, then $\phi$ is also an endomorphism of $(X,\cl_C)$.

\begin{corollary}\label{cor:ideal}
Let $I$ be an ideal of $\N^m$, that is, a $\preceq$-downward closed subset of $\N^m$. For each $\sv\in \N^k$, let $H_I(\sv)$ be the number of elements $\rv \in I$ with $\Vert\rv\Vert=\sv$. Then there is a polynomial $P \in \Q[\Y]$ such that
\[
H_I(\sv)=P(\sv)
\]
whenever $\min(\sv)$ is sufficiently large.
\end{corollary}
\begin{proof}
Let $C\coloneqq \N^m \setminus I$, let $(X,\cl)$ be the set $\N^m$ with the trivial closure operator, and let $(X,\cl_C)$ be the relativization of $(X,\cl)$ at $C$. Then $\rk_C(Y)= |Y \setminus C| =| Y \cap I|$ for any subset $Y\subseteq \N^m$. For each $i \in \{1,\ldots,m\}$, let $\phi_i\colon \N^m\to \N^m$ be the map $\rv \mapsto\rv+\eev_{i,m}$, so
\[
\rk_C(\Phi^{(\sv)}(\zero_m))=|\{\rv \in \N^m:\Vert \rv \Vert =\sv\}\cap I |=H_I(\sv).
\]
Since $I$ is $\preceq$-downward closed, we have $\phi_i(C) \subseteq C$ for each $i \in \{1,\ldots,m\}$. Thus, $\phi_1,\ldots,\phi_m$ are commuting endomorphisms of $(X,\cl_C)$, and the corollary follows from Theorem~\ref{thm:hilbert}.
\end{proof}

The following variant of Corollary~\ref{cor:ideal} was established by Kondratieva, Levin, Mikhalev, and Pankratiev~\cite{KLMP99}, with the $k = 1$ case first proven by Kolchin~\cite[Chapter 0, Lemma 16]{Ko73}. This variant can be deduced in the same way as Corollary~\ref{cor:ideal}; just use Corollary~\ref{cor:kolchin} in place of Theorem~\ref{thm:hilbert}:

\begin{corollary}[Theorem 2.2.5 in~\cite{KLMP99}]\label{cor:ideal2}
Let $I$ be an ideal of $\N^m$, and for each $\sv\in \N^k$, let $H^*_I(\sv)$ be the number of elements $\rv \in I$ with $\Vert\rv\Vert\preceq \sv$. Then there is a polynomial $Q \in \Q[\Y]$ such that
\[
H^*_I(\sv)=Q(\sv)
\]
whenever $\min(\sv)$ is sufficiently large.
\end{corollary}

\subsection{The Hilbert polynomial}
Let $K$ be a field and let $R\coloneqq K[x_1,\ldots,x_m]$, where $x_1,\ldots,x_m$ are 
indeterminates. 
Using our partition $(\Phi_1,\ldots,\Phi_k)$, we associate to $R$ an $\N^k$-grading $R =\bigoplus_{\sv \in \N^k}R_{\sv}$ as follows:\ the graded part $R_{\sv}$ is the $K$-vector space generated by monomials $x_1^{r_1} \cdots x_m^{r_m}$ with $\Vert \rv\Vert =\sv$. Let $M=\bigoplus_{\sv \in \Z^k} M_\sv$ be a finitely generated graded $R$-module. Then each graded piece $M_\sv$ is a finite-dimensional $K$-vector space. The \textbf{Hilbert function of $M$}, denoted $H_M\colon \Z^k\to \N$, is given by $H_M(\sv)\coloneqq\dim_K(M_{\sv})$. The following result is classical:

\begin{corollary}\label{cor:classichilbert}
There is a polynomial $P \in \Q[\Y]$ such that
\[
H_M(\sv)=P(\sv)
\]
whenever $\min(\sv)$ is sufficiently large.
\end{corollary}
\begin{proof}
Let $(X,\cl)$ be the underlying set of our $R$-module $M$ with the $K$-linear closure operator (that is, the closure of any subset of $M$ is its $K$-linear span). For each $i \in \{1,\ldots,m\}$, let $\phi_i\colon M\to M$ be the map $a\mapsto x_i\cdot a$. Since $M$ is finitely generated, we can find an index $\sv_0$ and a finite set $A \subseteq M_{\sv_0}$ such that the $R$-submodule $\bigoplus_{\sv\succeq \sv_0} M_\sv$ is generated by $A$. By re-indexing, we may assume that $\sv_0 =\zero_k$. Then for each $\sv\in \N^k$, the graded part $M_{\sv}$ is generated as a $K$-vector space by $\Phi^{(\sv)}(A)$, so
\[
\dim_K(M_{\sv})=\rk(M_{\sv})=\rk(\Phi^{(\sv)}(A)).
\]
The corollary follows from Theorem~\ref{thm:hilbert}.
\end{proof}

While we deduce Corollary~\ref{cor:classichilbert} from Theorem~\ref{thm:hilbert}, one can also deduce Theorem~\ref{thm:hilbert} from Corollary~\ref{cor:classichilbert} using the decreasing function $f^\Phi_{A|B}$ that we constructed and exploited in Subsection~\ref{subsec:proof}:

\begin{proposition}\label{prop:whichonesappear}
Let $A,B$, and $\Phi = (\Phi_1,\ldots,\Phi_k)$ be as in Theorem~\ref{thm:hilbert}. There is a finitely generated graded $R$-module $M = \bigoplus_{\sv \in \N^k}M_{\sv}$ such that $\dim_KM_{\sv} = \rk(\Phi^{(\sv)}(A)|\Phi^{(\sv)}(B))$ for all $\sv \in \N^k$.
\end{proposition}
\begin{proof}
Let $f^\Phi_{A|B}\colon \N^m\to \N$ be as in the proof of Theorem~\ref{thm:hilbert}, so $f^\Phi_{A|B}$ is decreasing and bounded above by $|A|$. For each $n \geq 1$, we put $I_n = \{\uv \in \N^m:f^\Phi_{A|B}(\uv)\geq n\}$, so $I_n$ is an ideal of $\N^m$, as defined in Corollary~\ref{cor:ideal}, and $I_n = \emptyset$ whenever $n>|A|$. Let $\chi_n$ be the indicator function of $I_n$, so $f^\Phi_{A|B} = \chi_1+\chi_2+\cdots$ and for $\sv \in \N^k$, we have 
\[
\rk(\Phi^{(\sv)}(A)|\Phi^{(\sv)}(B)) = \sum_{\Vert \uv\Vert =\sv}f^\Phi_{A|B}(\uv) = \sum_{n\geq 1} \sum_{\Vert \uv\Vert =\sv}\chi_n(\uv).
\]
For each $n\geq 1$, let $M_n \coloneqq R/(\bar{X}^{\uv}:\uv \not\in I_n)$, so $M_n = \bigoplus_{\sv \in \N^k}(M_n)_{\sv}$ is a graded $R$-module with
\[
\dim_K(M_n)_{\sv} = |\{\uv \in I_n:\Vert \uv\Vert = \sv\}| = \sum_{\Vert \uv\Vert =\sv}\chi_n(\uv).
\]
Let $M \coloneqq \bigoplus_{n \geq 1}M_n$. Then
\[
\dim_KM_{\sv} = \sum_{n\geq 1}\dim_K(M_n)_{\sv} = \rk(\Phi^{(\sv)}(A)|\Phi^{(\sv)}(B)).\qedhere
\]
\end{proof}

Much work has been done on characterizing exactly which numerical polynomials arise as Hilbert polynomials of finitely generated graded $R$-modules; see~\cite{KMU20} for our present multivariate setting. It follows from the above proposition that we don't obtain any new dimension polynomials in our framework:

\begin{corollary}\label{cor:whichonesappear}
The dimension polynomials we obtain in Theorem~\ref{thm:hilbert} are exactly the Hilbert polynomials of finitely generated graded $R$-modules, as in Corollary~\ref{cor:classichilbert}.
\end{corollary}

\subsection{Tropical ideals}
Let $\bar{\R}$ be the tropical semiring $\R\cup \{\infty\}$ with operations $a \oplus b
\coloneqq \min(a,b)$ and $a \odot b = a+b$. 
Let $\bar{\R}[\x]=\bar{\R}[x_1,\ldots,x_m]$ consist of the tropical polynomials $f(\x) = \bigoplus_\rv f_\rv\odot\x^\rv$. Then $\bar{\R}[\x]$ is a tropical semiring in its own right, where $\oplus$ and $\odot$ are extended to $\bar{\R}[\x]$ in the natural way. We associate to each tropical polynomial $f$ the subset of monomials $\supp(f)\coloneqq\{\x^{\rv}:f_\rv\neq \infty\}$.

\smallskip
An \textbf{ideal} is a subset $I\subseteq \bar{\R}[\x]$ which is closed under sums and multiplication by elements of $\bar{\R}[\x]$. Following Maclagan and Rinc\'{o}n~\cite{MR18}, we say that an ideal $I$ is \textbf{tropical} if it satisfies the `monomial elimination axiom':\ for any $f, g \in I$ and any $\rv\in \N^m$ for which $f_\rv = g_\rv\neq \infty$, there is $h\in I$ such that $h_{\rv} = \infty$ and $h_{\uv} \geq f_{\uv}\oplus g_{\uv}$ for all $\uv\in \N^m$, with equality whenever $f_{\uv} \neq g_{\uv}$. Using our partition $(\Phi_1,\ldots,\Phi_k)$, we associate to $\bar{\R}[\x]$ an $\N^k$-grading, where the graded part $\bar{\R}[\x]_{\sv}$ consists of all homogeneous tropical polynomials of degree $\sv$. We call an ideal $I$ \textbf{homogeneous} if it is homogeneous with respect to this grading.

\smallskip
Let $I$ be a homogeneous tropical ideal. For $\sv \in \N^k$, let $I_{\sv} = I \cap \bar{\R}[\x]_{\sv}$. Let $\Mon_\sv \coloneqq \{\x^\rv:\Vert\rv\Vert = \sv\}$ be the set of monomials of degree $\sv$, and define the closure operator $\cl_I\colon \cP(\Mon_{\sv})\to \cP(\Mon_{\sv})$ as follows:\ for $A \subseteq \Mon_{\sv}$, let $\cl_I(A)$ consist all elements of $A$ along with all $\x^\rv\in \Mon_{\sv}$ such that $A_0\cup\{\x^{\rv}\} = \supp(f)$ for some $A_0 \subseteq A$ and some $f \in I_{\sv}$.

\begin{lemma}
$(\Mon_\sv,\cl_I)$ is a finite matroid.
\end{lemma}
\begin{proof}
Monotonicity is clear. For idempotence, let $\x^{\rv} \in \cl_I(\cl_I(A))$, and take $f \in I_{\sv}$ with $\supp(f) = B\cup \{\x^{\rv}\}$ for some $B \subseteq \cl_I(A)$. We may assume that $|B \setminus A|$ is minimal, and we want to show that $|B \setminus A| = 0$. Suppose that this is not the case, let $\x^{\uv} \in B \setminus A$, and take $g \in I_{\sv}$ with $\supp(g) = A_0\cup\{\x^{\uv}\}$ for some $A_0 \subseteq A$. By multiplying $g$ with an appropriate element of $\bar{\R}$, we may arrange that $g_{\uv} = f_{\uv}$. By the monomial elimination axiom, there is $h\in I_{\sv}$ such that $\supp(h) = B^*\cup\{\x^{\rv}\}$ for some $B^* \subseteq (B\cup A_0)\setminus \{\x^{\uv}\}$. But then $|B^* \setminus A|< |B \setminus A|$, contradicting our choice of $B$. Finally, for exchange, let $\x^{\rv},\x^{\uv} \in \Mon_{\sv}$ and $A \subseteq \Mon_{\sv}$ with $\x^{\rv} \in \cl_I(\{\x^{\uv} \}\cup A) \setminus \cl_I(A)$. Take $f \in I_{\sv}$ with $\supp(f) =B \cup \{\x^{\rv} \}$ for some $B \subseteq \{\x^{\uv}\}\cup A$. If $B \subseteq A$, then we would have $\x^{\rv} \in \cl_I(A)$, so $B$ must contain $\x^{\uv}$. Then $\x^{\uv} \in \cl_I(\{\x^{\rv} \}\cup A)$, as witnessed by $f$.
\end{proof}

\begin{corollary}\label{cor:tropical}
Let $I$ be a homogeneous tropical ideal. There is a polynomial $P \in \Q[\Y]$ such that
\[
\rk_I(\Mon_{\sv})=P(\sv)
\]
whenever $\min(\sv)$ is sufficiently large, where $\rk_I$ is the rank associated to $\cl_I$.
\end{corollary}
\begin{proof}
Let $\Mon\coloneqq \{\x^{\rv}:\rv \in \N^m\}$ be the collection of all monomials. We extend the closure operators on the sets $\Mon_{\sv}$ to a closure operator $\cl_I$ on $\Mon$ as follows:
\[
\x^{\rv} \in \cl_I(A) :\Longleftrightarrow \x^{\rv} \in \cl_I(A \cap \Mon_{\Vert \rv\Vert}).
\]
Then $(\Mon,\cl_I)$ is the disjoint union of the finite matroids $(\Mon_{\sv},\cl_I)$, so $(\Mon,\cl_I)$ is a finitary (infinite) matroid, as can be easily verified. It is also routine to check that the map $\x^\rv \mapsto x_i\odot\x^{\rv}$ is an endomorphism of $(\Mon,\cl_I)$. Note that 
\[
\rk_I(\Mon_{\sv})=\rk_I(\Phi^{(\sv)}(\x^{\zero_m})),
\]
so the corollary follows from Theorem~\ref{thm:hilbert}.
\end{proof}

\begin{remark}
Maclagan and Rinc\'{o}n established Corollary~\ref{cor:tropical} in the case $k = 1$~\cite{MR18}. In their definition of a tropical ideal, the underlying matroid is defined in terms of circuits (minimal dependent sets). The circuits are exactly the minimal nonempty elements of the set $\{\supp(f):f \in I_{\sv}\}$, ordered by inclusion. In fact, Maclagan and Rinc\'{o}n equip the underlying sets $\Mon_{\sv}$ with the richer structure of a \emph{valuated} matroid (though the Hilbert polynomial is defined in terms of the underlying classical matroid).
\end{remark}

\section{Applications II:\ Simplicial complexes and Betti numbers}\label{sec:betti}
\noindent
Let $\cK$ be a simplicial complex. In this section, we consider commuting simplicial endomorphisms $\phi_1,\ldots,\phi_m$ of $\cK$. We will use Theorem~\ref{thm:hilbert} to show that for a finite subcomplex $A \subseteq \cK$, the Betti numbers $b_n(\Phi^{(\sv)}(A))$ are all eventually polynomial in $\sv$.

\smallskip
Let $(C_\bullet,\partial_\bullet)$ be the chain complex associated to $\cK$, so $C_n$ is the free abelian group generated by the $n$-simplices in $\cK$ and $\partial_n\colon C_n\to C_{n-1}$ is the boundary map. For each $n$, we define two ranks on $C_n$ as follows:\ given a finite subset $B \subseteq C_n$, we let $\rk_n(B)$ be the rank of the group generated by $B$, and we let $\rk_n^\partial(B)$ be the rank of the image of this group under $\partial_n$. Note that $\rk_0^\partial$ is always zero.

\smallskip
Given a finite subcomplex $A$ of $\cK$, we have an associated chain subcomplex $(C_\bullet(A),\partial_\bullet) \subseteq (C_\bullet,\partial_\bullet)$. We write $\rk_n(A)$ and $\rk_n^\partial(A)$ in place of $\rk_n(C_n(A))$ and $\rk_n^\partial(C_n(A))$. Then $\rk_n(A)- \rk_n^\partial(A)$ is the rank of the kernel of $\partial_n\colon C_n(A)\to C_{n-1}(A)$, and $\rk_{n+1}^\partial(A)$ is the rank of the image of $\partial_{n+1}\colon C_{n+1}(A)\to C_{n}(A)$. It follows that the $n$-th Betti number of $A$ --- that is, the rank of the $n$-th simplicial homology group $H_n(A)$ --- is exactly $\rk_n(A)- \rk_n^\partial(A)-\rk_{n+1}^\partial(A)$. It is routine to verify the following:

\begin{lemma}
For each $n$, the functions $\rk_n$ and $\rk_n^\partial$ both are rank functions of finitary matroids on $C_n$. That is, if we define the closure operators $\cl_n$ and $\cl_n^\partial$ on $C_n$ by putting 
\begin{align*}
a \in \cl_n(B) &:\Longleftrightarrow \rk_n(B_0\cup\{a\}) = \rk_n(B_0) \text{ for some finite }B_0\subseteq B\\
a \in \cl_n^\partial(B) &:\Longleftrightarrow \rk_n^\partial(B_0\cup\{a\}) = \rk_n^\partial(B_0) \text{ for some finite }B_0 \subseteq B,
\end{align*}
then $(C_n,\cl_n)$ is a finitary matroid with corresponding rank function $\rk_n$, and likewise for $(C_n,\cl_n^\partial)$.
\end{lemma}

Let $V(\cK)$ denote the vertex set of $\cK$ (that is, the set of
zero-dimensional simplices). A \textbf{simplicial endomorphism} of $\cK$ is a
map $f\colon V(\cK)\to V(\cK)$ that maps every simplex in $\cK$ to a (possibly
lower-dimensional) simplex in $\cK$. Explicitly, if $\sigma\in \cK$ is an $n$-simplex
with vertex set $\{v_0,\ldots,v_n\}\subseteq V(\cK)$, then we require that
$\{f(v_0),\ldots,f(v_n)\}$ be the vertex set of a simplex in $\cK$, which we denote
$f(\sigma)$. 
If $A$ is a subcomplex of $\cK$, then $f(A)\coloneqq\{f(\sigma):\sigma\in A\}$ is
also a subcomplex of $\cK$. 
Any simplicial endomorphism of $\cK$ induces a chain endomorphism $f_\bullet\colon(C_\bullet,\partial_\bullet)\to(C_\bullet,\partial_\bullet)$. Thus, we have the following:

\begin{lemma}
Let $f$ be a simplicial endomorphism of $\cK$, and let $f_\bullet\colon(C_\bullet,\partial_\bullet)\to(C_\bullet,\partial_\bullet)$ be the chain endomorphism induced by $f$. Then $f_n$ is an endomorphism of the finitary matroids $(C_n,\cl_n)$ and $(C_n,\cl_n^\partial)$ for each $n$.
\end{lemma}

We can now state our result on the growth of Betti numbers:

\begin{corollary}\label{cor:betti}
Let $A$ be a finite subcomplex of $\cK$ and suppose that each $\phi_i$ is a simplicial endomorphism of $\cK$. Then for each $n$, the $n$-th Betti number $b_n(\Phi^{(\sv)}(A))$ is eventually polynomial in $\sv$.
\end{corollary}
\begin{proof}
For a fixed $n$, we have $b_n(\Phi^{(\sv)}(A))= \rk_n(\Phi^{(\sv)}(A)) - \rk_n^\partial(\Phi^{(\sv)}(A))-\rk_{n+1}^\partial(\Phi^{(\sv)}(A))$. These three ranks are eventually polynomial in $\sv$ by Theorem~\ref{thm:hilbert}, so $b_n(\Phi^{(\sv)}(A))$ is as well.
\end{proof}

\subsection{Topological dynamical systems}\label{subsec:counterexample1}
Let $B$ be a topological space, let $\phi_1,\ldots,\phi_m$ be commuting continuous maps $B\to B$, and let $A$ be a compact subspace of $B$. We say that the topological dynamical system $(B,A;\phi_1,\ldots,\phi_m)$ is \textbf{triangulable} if there is a simplicial complex $\cK$ with underlying space $|\cK|$, a finite subcomplex $\cK_0 \subseteq \cK$, simplicial maps $f_1,\ldots,f_m\colon \cK\to \cK$, and a homeomorphism $\tau\colon |\cK|\to B$ such that $\tau(|\cK_0|) = A$ and such that $\tau\circ f_i = \phi_i\circ \tau$ for each $i$. If $(B,A;\phi_1,\ldots,\phi_m)$ is triangulable, then $n$-th Betti number $b_n(\Phi^{(\sv)}(A))$ is eventually polynomial in $\sv$ by Corollary~\ref{cor:betti}.

\smallskip
This leads to the question:\ if $(B, A;\phi_1,\ldots,\phi_m)$ is a non-triangulable topological dynamical system, can the Betti numbers $b_n(\Phi^{(\sv)}(A))$ fail to grow eventually polynomially? The following example shows that even for relatively nice topological dynamical systems, this can fail. Let $A\subseteq \R^2$ be the circle of radius $\sqrt{2}/2$, centered at the point $(0,1)$. Take $\alpha \in (0,\pi/2)$ such that $\alpha/\pi$ is irrational, and let $\phi\colon \R^2\to \R^2$ be counterclockwise rotation by $\alpha$ about the origin. We consider the (clearly non-triangulable) system $(\R^2,A;\id,\phi)$, pictured below, where $\id$ is the identity map. 
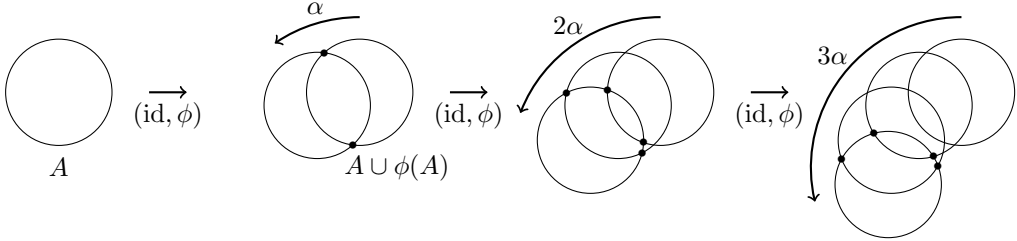
\begin{center}
\begin{tikzpicture}
\node at (-4,0) {$A$};
\draw(0-4,1) circle (0.70710) ;

\draw[thick,above, ->] (-2.8,1) to node[below]{$(\id,\phi)$} (-2.3,1);

\node at (0.5,0) {$A\cup \phi(A)$};
\node at (-0.472161,1.52637) {\tiny $\bullet$};
\node at (-0.0924813,0.298967) {\tiny $\bullet$};
\draw(0,1) circle (0.70710);
\draw(-0.564642, 0.825336) circle (0.70710);
\draw[thick, ->] (0,2) arc (90:90+34.38:2) node[midway,above] {$\alpha$};

\draw[thick,above, ->] (1.2,1) to node[below]{$(\id,\phi)$} (1.7,1);

\node at (4-0.70636,1.03248) {\tiny $\bullet$};
\node at (4-0.225679,0.329874) {\tiny $\bullet$};
\node at (4-1.25154,0.993164) {\tiny $\bullet$};
\node at (4-0.245138,0.194529) {\tiny $\bullet$};
\draw(4,1) circle (.707);
\draw(4-0.564642, 0.825336) circle (.707);
\draw(4-0.932039, 0.362358) circle (.707);
\draw[thick, ->] (4,2) arc (90:90+34.38+34.38:2) node[midway,above] {$2\alpha\phantom{..}$};

\draw[thick,above, ->] (5.2,1) to node[below]{$(\id,\phi)$} (5.7,1);

\node at (8-1.16597, 0.453305) {\tiny $\bullet$};
\node at (8-0.372522, 0.144829) {\tiny $\bullet$};
\node at (8-1.59373, 0.113019) {\tiny $\bullet$};
\node at (8-0.31216, 0.0221368) {\tiny $\bullet$};
\draw(8,1) circle (.707);
\draw(8-0.564642, 0.825336) circle (.707);
\draw(8-0.932039, 0.362358) circle (.707);
\draw(8-0.973848, -0.227202) circle (.707);
\draw[thick, ->] (8,2) arc (90:90+34.38+34.38+34.38:2) node[midway,above] {$3\alpha\phantom{...}$};
\end{tikzpicture}
\end{center}

\smallskip
For each $t$, put $A_t\coloneqq A\cup\cdots \cup \phi^t(A)$. We apply Mayer-Vietoris to $A_{t+1} = A_t \cup \phi^{t+1}(A)$, using that $H_1(A_t \cap \phi^{t+1}(A)) = 0$, to get an exact sequence
\[
0 \to H_1(A_t)\oplus H_1(A) \to H_1(A_{t+1}) \to H_0(A_t \cap \phi^{t+1}(A)) \to H_0(A_t)\oplus H_0(A) \to H_0(A_{t+1}) \to 0.
\]
With the exception of $H_1(A_{t+1})$, $H_1(A_t)$, and $H_0(A_t \cap \phi^{t+1}(A))$, all the homology groups above have rank 1, so we compute that
\begin{equation}\label{eq:betti}
b_1(A_{t+1}) - b_1(A_t)=b_0(A_t \cap \phi^{t+1}(A)).
\end{equation}
The intersection $A_t \cap \phi^{t+1}(A)$ is indicated by dots in the picture above. To determine the number of connected components in this intersection (that is, zeroth Betti number), let us first compute the intersection of two circles $\phi^i(A)$ and $\phi^j(A)$. If the angle between these two circles is less than $\pi/2$ in absolute value, then $\phi^i(A)$ and $\phi^j(A)$ intersect in exactly two points. If the angle is larger than $\pi/2$, then $\phi^i(A)$ and $\phi^j(A)$ are disjoint. The case that $\phi^i(A)$ and $\phi^j(A)$ intersect in exactly one point never occurs, as $\alpha/\pi$ is irrational. It follows that for $t$ large enough, the Betti number $b_0(A_t \cap \phi^{t+1}(A))$ is an even number which is approximately equal to $t+1$. In particular $b_0(A_t \cap \phi^{t+1}(A))$ is \emph{not} eventually polynomial in $t$. Using (\ref{eq:betti}), we see that $b_1(A_t)$ can not be eventually polynomial in $t$ either. 

\subsection{The graphic matroid}\label{subsec:counterexample}
Let $\cG$ be an undirected graph with vertex set $V(\cG)$ and edge set $E(\cG)$. Then $\cG$ is a simplicial complex with only zero and one-dimensional simplices. The rank $\rk_0$ corresponds to the cardinality of a set of vertices, the rank $\rk_1$ gives the cardinality of a set of edges, and $\rk_1^\partial$ is the rank with respect to the graphic matroid on $E(\cG)$. It follows from Corollary~\ref{cor:betti} that for a finite subgraph $A \subseteq \cG$ and commuting graph endomorphisms $\phi_1,\ldots,\phi_m$ of $\cG$, the first two Betti numbers $b_0(\Phi^{(\sv)}(A))$ and $b_1(\Phi^{(\sv)}(A))$ are eventually polynomial in $\sv$.

\smallskip
Using the graphic matroid, we can construct an example that shows that our assumption in Theorem~\ref{thm:hilbert} that each $\Phi_i$ is triangular is, in some sense, necessary. Consider the following graph $\cG$:

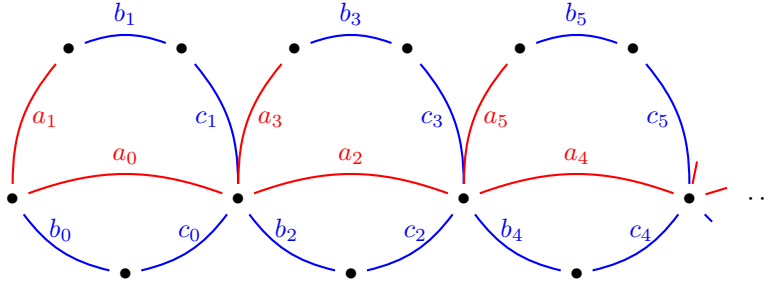
\begin{center}
\begin{tikzpicture}
	\node(A) at (0,0){$\bullet$};
	\node(B) at (.75,2){$\bullet$};
	\node(B') at (1.5,-1){$\bullet$};
	\node(C) at (3,0){$\bullet$};
	\node(C') at (2.25,2){$\bullet$};
	\node(D) at (3.75,2){$\bullet$};
	\node(D') at (4.5,-1){$\bullet$};
	\node(E) at (6,0){$\bullet$};
	\node(E') at (5.25,2){$\bullet$};
	\node(F) at (6.75,2){$\bullet$};
	\node(F') at (7.5,-1){$\bullet$};
	\node(G) at (9,0){$\bullet$};
	\node(G') at (8.25,2){$\bullet$};
	\node(H) at (10,0){$\cdots$};
	\tikzset{EdgeStyle/.style = {bend left=20}}
	
	\draw[EdgeStyle,thick,above,red](A) to node{$a_0$} (C) ;
	\draw[EdgeStyle,thick,above,blue](C) to node{$c_0$} (B') ;
	\draw[EdgeStyle,thick,above,blue](B') to node{$b_0$} (A);
	\draw[EdgeStyle,thick,right,red](A) to node{$a_1$} (B) ;
	\draw[EdgeStyle,thick,above,blue](B) to node{$b_1$} (C');
	\draw[EdgeStyle,thick,left,blue](C') to node{$c_1$} (C);
	
	\draw[EdgeStyle,thick,above,red](C) to node{$a_2$} (E) ;
	\draw[EdgeStyle,thick,above,blue](E) to node{$c_2$} (D') ;
	\draw[EdgeStyle,thick,above,blue](D') to node{$b_2$} (C);
	\draw[EdgeStyle,thick,right,red](C) to node{$a_3$} (D) ;
	\draw[EdgeStyle,thick,above,blue](D) to node{$b_3$} (E');
	\draw[EdgeStyle,thick,left,blue](E') to node{$c_3$} (E);
	
	\draw[EdgeStyle,thick,above,red](E) to node{$a_4$} (G) ;
	\draw[EdgeStyle,thick,above,blue](G) to node{$c_4$} (F') ;
	\draw[EdgeStyle,thick,above,blue](F') to node{$b_4$} (E);
	\draw[EdgeStyle,thick,right,red](E) to node{$a_5$} (F) ;
	\draw[EdgeStyle,thick,above,blue](F) to node{$b_5$} (G');
	\draw[EdgeStyle,thick,left,blue](G') to node{$c_5$} (G);

	\draw[thick,red](G) to node{} (9.5,.15) ;
	\draw[thick,blue](9.3,-.3) to node{} (G);
	\draw[thick,red](G) to node{} (9.1,.5) ;

\end{tikzpicture}
\end{center}

Let $\phi$ be the map sending $a_i$ to $a_{i+1}$, sending $b_i$ to $b_{i+1}$, and sending $c_i$ to $c_{i+1}$ for each $i$. Let $\cl$ be the closure operator associated to the graphic matroid on $E(\cG)$ (explicitly, an edge $a$ is in the closure of a set of edges $B$ if there are $b_1,\ldots,b_n \in B$ which connect the endpoints of $a$), and let $\rk$ be the corresponding rank function. One can check that $\phi a \in \cl(B\phi B)$ whenever $a \in \cl(B)$, so $\phi$ is a quasi-endomorphism of the graphic matroid $E(\cG)$. We have
\[
\rk(\phi^t(a_0,b_0,c_0))=\rk(a_t,b_t,c_t)=\left\{
\begin{array}{ll}
2 & \mbox{if $t$ is even}\\
3 & \mbox{if $t$ is odd},
\end{array}\right.
\]
so the conclusion of Theorem~\ref{thm:hilbert} fails for the single map $\phi$. Of course, since $\phi$ is a quasi-endomorphism, Corollary~\ref{cor:kolchin} tells us that $\rk(\{a_i,b_i,c_i:i\leq t\})$ is eventually polynomial in $t$, as can easily be verified.

\section{Applications III:\ Difference-differential fields}\label{sec:dfields}
\noindent
In this final section, we examine endomorphisms of and derivations on (expansions of) fields. Recall that a \textbf{derivation} on a field $K$ is a map $\delta\colon K\to K$ that satisfies the identities
\[
\delta(a+b)=\delta a + \delta b,\qquad \delta(ab)=a\delta b + b\delta a
\]
for all $a,b \in K$.

\subsection{Kolchin polynomials for difference-differential fields}
Let $K$ be a field of characteristic zero. We let $\acl$ be algebraic closure in $K$, so for $A \subseteq K$, the set $\acl(A)$ consists of all $a \in K$ which are algebraic over the field $\Q(A)$. Then $(K,\acl)$ is a finitary matroid, and the rank of $A\subseteq K$ with respect to $\acl$ is equal to $\trdeg_\Q\Q(A)$, the transcendence degree of $\Q(A)$ over $\Q$. Clearly, any field endomorphism $\sigma\colon K\to K$ is an endomorphism of $(K,\acl)$.

\begin{lemma}\label{lem:derisquasi}
Let $\delta$ be a derivation on $K$. Then $\delta$ is a quasi-endomorphism of $(K,\acl)$. 
\end{lemma}
\begin{proof}
Let $a\in K$ and $B \subseteq K$ with $a \in \acl(B)$. Let $P$ be a polynomial of minimal degree witnessing this, so $P(a) = 0$ and $P'(a) \neq 0$. Write $P(X) = \sum_db_dX^d$, where each $b_d \in \Q[B]$. We have
\[
0 =\delta P(a)=\sum_d\delta(b_da^d) =\sum_d\Big(\delta(b_d)a^d +db_da^{d-1}\delta a\Big) = \sum_d\delta(b_d)a^d +P'(a)\delta a,
\]
so $\delta a\in \acl(aB\delta B)\subseteq \acl(B\delta B)$, as desired. 
\end{proof}

Suppose now that $\Phi = (\phi_1,\ldots,\phi_m)$ is a collection of commuting maps $K\to K$, each of which is either a derivation or a field endomorphism (hence, a quasi-endomorphism of $(K,\acl)$ by Lemma~\ref{lem:derisquasi}). Then $(K,\phi_1,\ldots,\phi_m)$ is called a \textbf{difference-differential field} (or a \textbf{$\dif$-field} for short). Let $F$ be a $\dif$-subfield of $K$, that is, a subfield of $K$ which is closed under each $\phi_i$. Then each $\phi_i$ is a quasi-endomorphism of the relativization $(K,\acl_F)$ of $(K,\acl)$, and the rank of a subset $A \subseteq K$ with respect to $\acl_F$ coincides with $\trdeg_FF(A)$. Applying Corollary~\ref{cor:kolchin}, Proposition~\ref{prop:invariant}, and Corollary~\ref{cor:kolchincoefficient} to $(K,\acl_F)$, we get the following:

\begin{corollary}\label{cor:ddfields}
Let $\av$ be a tuple from $K$. Then there is a polynomial $Q^\Phi_{\av} \in \Q[\Y]$ such that
\[
\trdeg_FF(\Phi^{\preceq(\sv)}(\av))=Q^\Phi_{\av}(\sv)
\]
whenever $\min(\sv)$ is sufficiently large. Moreover:
\begin{enumerate}[(i)]
\item The dominant terms of $Q^\Phi_{\av}$ only depend on the algebraic closure of the $\dif$-field extension $F(\Theta(\av))$, that is, if $\bv \in K$ and if $F(\Theta(\av))$ and $F(\Theta(\bv))$ have the same algebraic closure in $K$, then $Q^\Phi_{\av}$ and $Q^\Phi_{\bv}$ have the same dominant terms.
\item The coefficient of $\Y^\dv$ in $Q^\Phi_{\av}$ times $d_1!\cdots d_k!$ is equal to the maximal size of a subset $B\subseteq\acl_F(\Theta(\av))$ such that the tuple $(\theta b)_{\theta \in \Theta,b \in B}$ is algebraically independent over $F$.
\end{enumerate}
\end{corollary}

The case when $k= 1$ and each $\phi_i$ is a derivation was established by Kolchin~\cite{Ko64}. The multivariate differential case (each $\phi_i$ is a derivation but $k$ is arbitrary) was shown by Levin~\cite{Le01}. The case where each $\phi_i$ is an endomorphism is also due to Levin~\cite{Le04}, as is the most general case to date:\ where each $\phi_i$ may be either a derivation or an endomorphism, but each part of the partition $\Phi_i$ must consist of only derivations or endomorphisms; see~\cite{Le20}. Our Corollary~\ref{cor:ddfields} is slightly more general than the result in~\cite{Le20}, since the parts of the partition can consist of both derivations and endomorphisms.

\subsection{Modules over a difference-differential field}
Let $(K,\Phi)$ be a difference-differential field, and let $K[\Phi]$ denote the ring
of linear difference-differential operators over $K$, so $K[\Phi]$ consists of
elements $\sum_{\rv}\lambda_{\rv}\phi^{\rv}$ 
with finitely many nonzero coefficients~$\lambda_{\rv}$. 
The multiplication on $K[\Phi]$ satisfies the identity 
\[
\phi_i\lambda=\left\{
\begin{array}{ll}
\phi_i(\lambda)\phi_i & \mbox{ if $\phi_i$ is an endomorphism of $K$}\\
\phi_i(\lambda)+ \lambda\phi_i & \mbox{ if $\phi_i$ is a derivation on $K$},
\end{array}
\right.
\]
for $i \in \{1,\ldots,m\}$ and $\lambda \in K$. Let $V$ be a left $K[\Phi]$-module. For each $i$ and each $v \in V$, we let $\phi_i(v)\coloneqq \phi_i\cdot v$, so each $\phi_i$ is a map $V\to V$.

\begin{lemma}
Let $\cl$ denote the $K$-linear closure operator on $V$ and let $i \in \{1,\ldots,m\}$. If $\phi_i$ is an endomorphism of $K$, then the map $v \mapsto \phi_i(v)$ is an endomorphism of $(V,\cl)$. If $\phi_i$ is a derivation on $K$, then the map $v \mapsto \phi_i(v)$ is a quasi-endomorphism of $(V,\cl)$.
\end{lemma}
\begin{proof}
Let $v_1,\ldots,v_n\in V$ and suppose that $v$ is in the $K$-linear span of $v_1,\ldots,v_n$. Take scalars $\lambda_1,\ldots,\lambda_n \in K$ with $v = \sum_{j= 1}^n\lambda_j \cdot v_j$. If $\phi_i$ is an endomorphism of $K$, then 
\[
\phi_i(v) =\sum_{j= 1}^n\phi_i(\lambda_j \cdot v_j)=\sum_{j= 1}^n\phi_i(\lambda_j)\cdot \phi_i(v_j),
\]
so $\phi_i(v)$ is in the $K$-linear span of $\phi_i(v_1),\ldots,\phi_i(v_n)$. If $\phi_i$ is a derivation on $K$, then 
\[
\phi_i(v) =\sum_{j= 1}^n\phi_i(\lambda_j \cdot v_j)=\sum_{j= 1}^n\phi_i(\lambda_j)\cdot v_j+ \lambda_j\cdot \phi_i(v_j),
\]
so $\phi_i(v)$ is in the $K$-linear span of $v_1,\ldots,v_n,\phi_i(v_1),\ldots,\phi_i(v_n)$.
\end{proof}

We now associate to $K[\Phi]$ an $\N^k$-filtration as follows:\ for $\sv \in \N^k$, let $K[\Phi]_{\sv}$ consist of the linear operators $\sum_{\Vert\rv\Vert \preceq \sv}\lambda_{\rv}\phi_1^{r_1}\cdots\phi_m^{r_m}$. Let $V = \bigcup_{\sv \in \Z^k}V_{\sv}$ be a filtered $K[\Phi]$-module. Following Johnson~\cite{Jo69}, we say that the filtration $(V_\sv)_{\sv \in \Z^k}$ is \textbf{excellent} if each $V_{\sv}$ is finite-dimensional as a $K$-vector space and if there is $\sv_0 \in \Z^k$ such that 
\[
V_\sv = \{D(v):D \in K[\Phi]_{\sv-\sv_0}\text{ and }v \in V_{\sv_0}\}
\]
for all $\sv \succeq \sv_0$.

\begin{corollary}\label{cor:johnson}
Let $V=\bigcup_{\sv \in \Z^k}V_{\sv}$ be an excellently filtered $K[\Phi]$-module. There is a polynomial $Q \in \Q[\Y]$ such that
\[
\dim_K(V_\sv)=Q(\sv)
\]
whenever $\min(\sv)$ is sufficiently large.
\end{corollary}
\begin{proof}
Let $(X,\cl)$ be the underlying set of $V$ with the $K$-linear closure operator. Since $V$ is excellent, we can find an index $\sv_0$ and a finite set $A \subseteq V_{\sv_0}$ such that the $K[\Phi]$-module $\bigcup_{\sv\succeq \sv_0} V_\sv$ is generated by $A$. By re-indexing, we may assume that $\sv_0 =\zero_k$. Then for each $\sv\in \N^k$, the graded part $V_{\sv}$ is generated as a $K$-vector space by 
\[
\Phi^{\preceq(\sv)}(A)=\{\phi^{\rv}(a):\Vert \rv\Vert \preceq \sv\text{ and }a\in A\}.
\] 
Thus, $\dim_K(M_{\sv})=\rk(\Phi^{\preceq(\sv)}(A))$ for $\sv \in \N^k$. The corollary follows from Corollary~\ref{cor:kolchin}.
\end{proof}

\begin{remark}
In the case that each $\phi_i$ is a derivation on $K$ and $k = 1$, Corollary~\ref{cor:johnson} is due to Johnson~\cite{Jo69}, who used this result to establish the existence of the Kolchin polynomial. Variations of this result were later given by Levin; see~\cite{Le78} for the case where each $\phi_i$ is a field automorphism. The dimension polynomial for modules over rings of differential operators also appears in work of Bernshtein~\cite{Be71}.
\end{remark}

\subsection{Higher derivations and $\cD$-fields}
Let $K$ be a field. 
A \textbf{higher derivation on $K$} is a tuple $\Delta = (\delta_0,\delta_1,\ldots,\delta_m)$ where $\delta_0\colon K\to K$ is the identity map and where $\delta_i\colon K\to K$ is an additive map satisfying the identity
\[
\delta_i(xy)=\sum_{j\leq i}\delta_j(x)\delta_{i-j}(y)
\]
for each $i \leq m$; see~\cite[Section 27]{Ma86}. Note that $\delta_1$ is a classical derivation on $K$. We may relax the assumption that $\delta_0$ is the identity map and ask only that it is an endomorphism. In this case, we call $\Delta$ a \textbf{twisted higher derivation on $K$}. 

We will consider higher derivations as a particular case of $\cD$-fields in the
sense of \cite{MS14}. For the remainder of this subsection, fix a subfield $A$ of $K$, and let $\Delta = (\delta_0,\delta_1,\ldots,\delta_m)$ be a tuple of
$A$-linear maps from $K$ to~$K$. 
We say that 
$(K, \Delta)$ is a 
\textbf{twisted $\cD$-field} 
if $\delta_{0}$ is an endomorphism of $K$ and, for every
$k = 1, \ldots, m$, there are coefficients $a_{i,j,k} \in A$ such that
\begin{equation}\label{eq:mult-Dring}
\delta_{k}(xy) = \sum_{i,j\leq m} a_{i,j,k} \delta_{i}(x) \delta_{j}(y).
\end{equation}
We refer the reader to~\cite{MS14} for a more precise description of $\cD$-fields. We do note that in~\cite{MS14}, the coefficients $a_{i,j,k}$ only depend on the ring scheme $\cD$, and that the first map $\delta_0$ is required to be the identity map, not just a ring endomorphism. This is why we refer to $(K,\delta_0,\ldots,\delta_m)$ as a ``twisted'' $\cD$-field. 

We say that the twisted $\cD$-field $(K, \Delta)$ is \textbf{pyramidal} if the coefficients $a_{i,j,k}$ satisfy the following additional conditions:
\begin{enumerate}
\item $a_{k,0,k} = a_{0,k,k} = 1$;
\item if $i>0$, then $a_{k,i,k} = a_{i,k,k} = 0$;
\item if $i>k$ or $j>k$, then $a_{i,j,k} = 0$.
\end{enumerate}
If $(K, \Delta)$ is a pyramidal twisted $\cD$-field, then we can rewrite \eqref{eq:mult-Dring} as
\[
\delta_{0}(xy) = \delta_0(x)\delta_{0}(y),\qquad \delta_{k}(xy) = \delta_{k}(x) \delta_{0}(y) + \sum_{i,j<k} a_{i,j,k} \delta_{i}(x) \delta_{j}(y)+\delta_{0}(x) \delta_{k}(y)\quad\text{ for }k=1, \dotsc, m.
\]
The following lemma can be established via a straightforward induction on $k$ and $d$:

\begin{lemma}\label{lem:polyD}
Let $(K,\Delta)$ be a pyramidal twisted $\cD$-field. Then for each $0<k\leq m$ and each $d \in \N$, there is a polynomial $R_{k,d}$ over $A$ such that
\[
\delta_k(a^d)=d\delta_0(a)^{d-1}\delta_k(a)+ R_{k,d}(\delta_0a,\ldots,\delta_{k-1}a).
\]
for all $a \in K$.
\end{lemma}

\begin{lemma}\label{lem:higherderiv}
Let $(K,\Delta)$ be a pyramidal twisted $\cD$-field of characteristic zero. Let $a \in K$, and let $B \subseteq K$. If $a \in \acl_A(B)$, then $\delta_ka \in \acl_A(\delta_0B\cdots\delta_kB)$ for each $k\leq m$. 
\end{lemma}
\begin{proof}
Take a nonzero polynomial $P(X)$ of minimal degree witnessing that $a\in \acl_A(B)$, so $P(a) = 0$ and $P'(a)\neq 0$. Write $P(X) = \sum_db_dX^d$ where each $b_d \in A[B]$. We will show by induction on $k\leq m$ that $\delta_ka \in \acl_A(\delta_0B\cdots\delta_kB)$. This holds for $k = 0$ since $\delta_0$ is an endomorphism of $K$. Let $0<k\leq m$, and assume that $\delta_ja \in \acl_A(\delta_0 B\cdots \delta_jB)$ for all $j<k$. By Lemma~\ref{lem:polyD}, we have
\begin{align*}
0 &= \delta_kP(a)=\sum_d\delta_k(b_da^d) = \sum_d\Big(\delta_k(b_d)\delta_0(a^d)+ \sum_{i,j<k}a_{i,j,k}\delta_i(b_d)\delta_j(a^d)+\delta_0(b_d)\delta_k(a^d)\Big)\\
&= \sum_d\Big(\delta_k(b_d)\delta_0(a^d)+ \sum_{i,j<k}a_{i,j,k}\delta_i(b_d)\delta_j(a^d)+\delta_0(b_d)R_{k,d}(\delta_0a,\ldots,\delta_{k-1}a)+\delta_0(b_d)d\delta_0(a)^{d-1}\delta_k(a)\Big)\\
&= \sum_d\Big(\delta_k(b_d)\delta_0(a^d)+ \sum_{i,j<k}a_{i,j,k}\delta_i(b_d)\delta_j(a^d)+\delta_0(b_d)R_{k,d}(\delta_0a,\ldots,\delta_{k-1}a)\Big)+\delta_0(P'(a))\delta_k(a).
\end{align*}
Since $P'(a)\neq 0$, this shows that $\delta_ka$ is algebraic over $\delta_0B\cdots\delta_kB$, together with $\delta_j(a^d)$ for $j<k$. Applying Lemma~\ref{lem:polyD} again and using our induction hypothesis, we see that 
\[
\delta_j(a^d) \in \acl_{A}(\delta_0a,\ldots,\delta_ja)\subseteq \acl_{A}(\delta_0B\cdots \delta_kB)
\]
for $j <k$, so $\delta_ka \in \acl_{A}(\delta_0B\cdots \delta_kB)$.
\end{proof}

\begin{corollary}\label{cor:hdfields}
Let $K$ be a field of characteristic zero, let $F$ be a subfield of $K$ containing $A$, let $\av$ be a tuple from $K$, and let $\Phi = (\Phi_1,\ldots,\Phi_k)$ be a partitioned set of commuting maps $K\to K$. Suppose for each $i=1,\ldots,k$ that $(K,\Phi_i)$ is a pyramidal twisted $\cD_i$-field, and that $(F,\Phi_i)$ is a $\cD_i$-subfield of $K$. Then there is a polynomial $P^\Phi_{\av} \in \Q[\Y]$ such that
\[
\trdeg_FF(\Phi^{(\sv)}(\av))=P^\Phi_{\av}(\sv)
\]
whenever $\min(\sv)$ is sufficiently large.
\end{corollary}

\begin{question}
Does Corollary~\ref{cor:hdfields} hold for all $\cD$-fields (in the sense
of \cite{MS14}) with commuting operators?
\end{question}

\subsection{Higher derivations in positive characteristic}
Let $K$ be a field of characteristic $p>0$. Then any derivation $\delta$ on $K$ is trivial on the subfield $K^p \subseteq K$, as $\delta a^p = pa^{p-1} \delta a = 0$ for $a \in K$. It follows that $\delta$ is not generally a quasi-endomorphism of $(K,\acl)$. For example, let $K = \F_p(x,y)$ with $x,y$ mutually transcendental over $\F_p$, and consider the derivation $\delta$ on $K$ that satisfies $\delta x = 1,\delta y = x$. We have $y \in \acl(y^p)$, but 
\[
\delta y = x \not\in \acl(y^p,\delta y^p) = \acl(y^p).
\]
One may attempt to remedy this by using the separable closure $\scl$ in place of
the algebraic closure, but then we run into a different issue:\ $(K,\scl)$ need
not be a finitary matroid. For example, with $K$ as above, we have $x^p \in
\scl(x)\setminus \scl(\emptyset)$ but $x \not\in \scl(x^p)$. 
This second issue can be solved by relativizing at $K^p$. 

Before checking that $(K,\scl_{K^p})$ is indeed a finitary matroid, we note that for any $P(X)\in K^p[X]$ and any $a \in K$, there is some polynomial $Q(X) \in K^p[X]$ of degree at most $p -1$ with $P(a) = Q(a)$ and $P'(a) = Q'(a)$. To obtain $Q$, we simply replace any monomial $X^{np+d}$ appearing in $P$ with $a^{np}X^d \in K^p[X]$. Thus, we have $a \in \scl_{K^p}(B)$ if and only if $P(a) = 0$ for some nonzero polynomial $P(X)\in K^p[B][X]$ of degree at most $p-1$.

\begin{lemma}\label{lem:sclmatroid}
$(K,\scl_{K^p})$ is a finitary matroid. 
\end{lemma}
\begin{proof}
Since monotonicity, idempotence, and finite character hold for $\scl$, and since these properties are preserved by relativization, we need only check Steinitz exchange. Let $a,b \in K$ and $A \subseteq K$, and suppose that $a \in \scl_{K^p}(Ab) \setminus \scl_{K^p}(A)$. Let $P(X,Y)$ be a nonzero polynomial over $K^p[A]$ of degree at most $p-1$ in both $X$ and $Y$ with $P(a,b) = 0$. We write 
\[
P(X,Y)=\sum_{d<p}Q_d(X)Y^d
\]
where each $Q_d(X)$ is a polynomial over $K^p[A]$ of degree at most $p-1$. As $P$ is nonzero, some $Q_d$ is nonzero, so $Q_d(a) \neq 0$ as $a \not\in \scl_{K^p}(A)$. Thus, $P(a,b) = 0$ and $P(a,Y)\neq 0$, so $b \in \scl_{K^p}(Aa)$.
\end{proof}

Given a subfield $F \subseteq K$, the rank $\rk(K|F)$ corresponding to the closure
$\scl_{K^p}$ is the cardinality of a $p$-basis for $F$ over $K$;
see~\cite[Section 26]{Ma86}. 

We return to the setting of pyramidal twisted $\cD$-fields. Again, we fix a subfield $A \subseteq K$ and a tuple $\Delta = (\delta_0,\ldots,\delta_m)$ of $A$-linear maps $K\to K$. Lemma~\ref{lem:polyD} still applies. Therefore, the following lemma can be proven in the same way as Lemma~\ref{lem:higherderiv} above, so long as one takes the polynomial $P$ in that proof to have degree at most $p-1$.

\begin{lemma}
Let $(K,\Delta)$ be a pyramidal twisted $\cD$-field of characteristic $p>0$. Let $a \in K$, and let $B \subseteq K$. If $a \in \scl_{K^p[A]}(B)$, then $\delta_ka \in \scl_{K^p[A]}(\delta_0B\cdots\delta_kB)$ for each $k\leq m$. In particular,
\begin{enumerate}[(i)]
\item Any field endomorphism of $K$ is an endomorphism of $(K,\scl_{K^p})$,
\item Any derivation on $K$ is a quasi-endomorphism of $(K,\scl_{K^p})$. 
\item Any commuting twisted higher derivation on $K$ is a triangular system with respect to the matroid $(K,\scl_{K^p})$. 
\end{enumerate}
\end{lemma}

Accordingly, we can prove a positive characteristic analog of
Corollary~\ref{cor:hdfields}. 

\begin{corollary}\label{cor:hdfields-p}
Let $K$ be a field of characteristic $p>0$, let $F$ be a subfield of $K$ containing $K^p[A]$, let $\av$ be a tuple from $K$, and let $\Phi = (\Phi_1,\ldots,\Phi_k)$ be a partitioned set of commuting maps $K\to K$. Suppose for each $i=1,\ldots,k$ that $(K,\Phi_i)$ is a pyramidal twisted $\cD_i$-field, and that $(F,\Phi_i)$ is a $\cD_i$-subfield of $K$. Then there is a polynomial $P^\Phi_{\av} \in \Q[\Y]$ such that the size of a $p$-basis for $F(\Phi^{(\sv)}(\av))$ is equal to $P^\Phi_{\av}(\sv)$, whenever $\min(\sv)$ is sufficiently large.
\end{corollary}

\subsection{Kolchin polynomials for difference-differential exponential fields}
Let $K$ be a field of characteristic zero. An \textbf{exponential on $K$} is a group homomorphism $E\colon A(K) \to K^\times$, where $A(K)$ is a divisible subgroup of the additive group of $K$. If $E$ is an exponential on $K$, then the pair $(K, E)$ is called an \textbf{exponential field}. A subfield $F$ of $K$ is an \textbf{exponential subfield} if $E(a) \in F$ for all $a \in A(F) \coloneqq A(K) \cap F$. The fields $\R$ and $\CC$ with their usual exponential functions are exponential fields (where the domain of the exponential is the entire field). 

\smallskip
Let $(K, E)$ be an exponential field. An \textbf{$E$-term} is a partial function given by arbitrary compositions of $E$ and polynomials over $\Z$. Model theoretically speaking, an $E$-term is a term in the language $(+,\cdot,-,0,1, E)$; to avoid partially defined functions, one may take $E$ to be identically zero away from $A(K)$. Let $B \subseteq K$. A tuple $\av = (a_1,\ldots,a_n)$ is said to be a \textbf{regular solution to a Khovanskii system over $B$} if there is a tuple $\bv = (b_1,\ldots,b_m)$ from $B$ and $E$-terms $t_1,\ldots,t_n$ in $n+m$ variables such that
\[
t_1(\av,\bv)=\cdots=t_n(\av,\bv)=0,\qquad \det\begin{pmatrix}
\frac{\partial t_1}{\partial X_1}(\av,\bv) & \cdots &\frac{\partial t_1}{\partial X_n}(\av,\bv)\\
\vdots & \ddots & \vdots\\
\frac{\partial t_n}{\partial X_1}(\av,\bv) & \cdots & \frac{\partial t_n}{\partial X_n} (\av,\bv)
\end{pmatrix} \neq 0.
\]
The \textbf{exponential-algebraic closure of $B$}, written $\ecl(B)$, consists of all components of any regular solution to a Khovanskii system over $B$. The exponential-algebraic closure was first defined by Macintyre~\cite{Ma96}, and Kirby later showed that $(K,\ecl)$ is always a finitary matroid~\cite[Theorem 1.1]{Ki10}, extending earlier work of Wilkie~\cite{Wi08}. If $F$ is an exponential subfield of $K$ and $A$ is a subset of $K$, then we let $F(A)^E$ denote the exponential subfield of $K$ generated by $F$ and $A$, and we define the \textbf{exponential transcendence degree of $F(A)^E$ over $F$}, denoted $\etrdeg_FF(A)^E$, to be the rank $\rk\big(F(A)^E\big|F\big) = \rk(A|F)$ given by the matroid $(K,\ecl)$.

\smallskip
An \textbf{exponential endomorphism of $K$} is a field endomorphism $\sigma\colon K \to K$ such that $\sigma E(a) = E(\sigma a)$ for all $a \in A(K)$. An \textbf{exponential derivation on $K$} is a derivation $\delta\colon K \to K$ which satisfies the identity $\delta E(a) = E(a) \delta a$ for all $a \in A(K)$.

\begin{lemma}
Any exponential endomorphism of $K$ is an endomorphism of $(K,\ecl)$. Any exponential derivation on $K$ is a quasi-endomorphism of $(K,\ecl)$.
\end{lemma}
\begin{proof}
Let $\sigma$ be an exponential endomorphism of $K$, let $\delta$ be an exponential derivation on $K$, and let $B\subseteq K$. If $a\in K$ is a component of a regular solution to a Khovanskii system over $B$, then $\sigma(a)$ is a component to a regular solution to a Khovanskii system over $\sigma(B)$, namely, the same Khovanskii system but with the parameters from $B$ replaced with the corresponding parameters from $\sigma(B)$. Thus, $\sigma$ is an endomorphism of $(K,\ecl)$.

To see that $\delta$ is a quasi-endomorphism of $(K,\ecl)$, we use~\cite[Theorem 1.1]{Ki10}, which states that $a \in K$ belongs to $\ecl(B)$ if and only if every exponential derivation on $K$ which vanishes on $B$ also vanishes at $a$. Suppose $a \in \ecl(B)$, and let $\epsilon$ be an exponential derivation on $K$ which vanishes on $B \cup \delta(B)$. We need to show that $\epsilon\delta a= 0$. Consider the map $\epsilon\delta-\delta \epsilon\colon K\to K$, where $(\epsilon\delta-\delta \epsilon)(y) =\epsilon\delta y-\delta \epsilon y$. It is routine to show that $\epsilon\delta-\delta \epsilon$ is an exponential derivation. For $b \in B$, we have $\epsilon\delta b - \delta\epsilon b= 0$, since $\epsilon$ vanishes on $B \cup \delta(B)$. Thus, $\epsilon\delta-\delta \epsilon$ vanishes on $B$, so it also vanishes at $a$. Since $\epsilon$ also vanishes at $a$, we see that
\[
0=\epsilon\delta a - \delta\epsilon a=\epsilon\delta a.\qedhere
\]
\end{proof}

Suppose now that $\Phi = (\phi_1,\ldots,\phi_m)$ is a collection of commuting maps $K\to K$, each of which is either an exponential derivation or an exponential endomorphism. The structure $(K,E, \phi_1,\ldots,\phi_m)$ is called a \textbf{difference-differential exponential field} (or a \textbf{$\dif$-exponential field} for short). Let $F$ be a $\dif$-exponential subfield of $K$, that is, an exponential subfield of $K$ which is closed under each $\phi_i$. Applying Corollary~\ref{cor:kolchin}, Proposition~\ref{prop:invariant}, and Corollary~\ref{cor:kolchincoefficient} to the relativized matroid $(K,\ecl_F)$ gives us the following:

\begin{corollary}\label{cor:eddfields}
Let $\av$ be a tuple from $K$. Then there is a polynomial $Q^\Phi_{\av} \in \Q[\Y]$ such that
\[
\etrdeg_FF(\Phi^{\preceq(\sv)}(\av))^E=Q^\Phi_{\av}(\sv)
\]
whenever $\min(\sv)$ is sufficiently large. Moreover:
\begin{enumerate}[(i)]
\item The dominant terms of $Q^\Phi_{\av}$ only depend on the exponential-algebraic closure of the $\dif$-field extension $F(\Theta(\av))$.
\item The coefficient of $\Y^\dv$ in $Q^\Phi_{\av}$ times $d_1!\cdots d_k!$ is equal to the maximal size of a subset $B\subseteq \ecl_F(\Theta(\av))$ such that the tuple $(\theta b)_{\theta \in \Theta,b \in B}$ is exponential-algebraically independent over $F$.
\end{enumerate}
\end{corollary}

\begin{remark}
With the obvious changes, Corollary~\ref{cor:eddfields} may be applied to $j$-fields. These fields, introduced in~\cite{Et18}, are equipped with partially defined functions which behave like the modular $j$-function. The relevant closure operator in this setting is the $j\!\operatorname{cl}$-closure, defined in~\cite{Et18}, and the tuple $\Phi$ should consist of commuting $j$-field endomorphisms and $j$-derivations, also defined in~\cite{Et18}. See~\cite{AEK22} for more on the $j$-closure operator. Thanks to Vincenzo Mantova for bringing this to our attention.
\end{remark}

\subsection{Derivations on o-minimal fields}
Let $T$ be an o-minimal theory extending the theory of real closed ordered fields, and let $K$ be a model of $T$; see~\cite{vdD98} for definitions and background. The \textbf{definable closure operator on $K$}, denoted $\dcl$, is given by
\[
a \in \dcl(B) :\Longleftrightarrow a = f(\bv)\text{ for some tuple $\bv$ from $B$ and some $\emptyset$-definable function $f$.}
\]
It is well-known that $(K,\dcl)$ is a finitary matroid, and we denote the corresponding rank function by $\rk_T$. Given an elementary substructure $F$ of $K$ and a set $A \subseteq K$, we let $F\langle A \rangle$ denote the definable closure of $F \cup A$ in $K$. Then $F\langle A \rangle$ is also a model of $T$, and $\rk_T(F\langle A\rangle|F) = \rk_T(A|F)$.

\smallskip
A \textbf{$T$-derivation on $K$} is a map $\delta\colon K\to K$ such that for each tuple $\av = (a_1,\ldots,a_n) \in K^n$ and each $\emptyset$-definable function $f$ which is $\cC^1$ in a neighborhood of $\av$, we have
\begin{equation}\label{eq:Tder}
\delta f(\av)=\frac{\partial f}{\partial Y_1}(\av) \delta a_1+ \cdots + \frac{\partial f}{\partial Y_n}(\av) \delta a_n.
\end{equation}
The study of $T$-derivations was initiated by the authors in~\cite{FK21} and was expanded on by the second author~\cite{Ka21}. 
The link between compatible derivations on o-minimal structures and definable closure in o-minimal structures has long been used by Wilkie and others; see~\cite{Wi08,BKW10}.
The following fact can be proven along the lines of Lemma~\ref{lem:derisquasi}, using (\ref{eq:Tder}) above:

\begin{fact}
Any $T$-derivation on $K$ is a quasi-endomorphism of $(K,\dcl)$.
\end{fact}

Now suppose that $\Phi = (\phi_1,\ldots,\phi_m)$ is a collection of commuting $T$-derivations on $K$. Then $(K, \phi_1,\ldots,\phi_m)$ is called a \textbf{$T$-differential field}. Let $F$ be a $T$-differential subfield of $K$, that is, an elementary substructure of $K$ which is closed under each $\phi_i$. Then each $\phi_i$ is a quasi-endomorphism of the relativization $(K,\dcl_F)$.

\begin{corollary}\label{cor:Tfields}
 Let $\av$ be a tuple from $K$. Then there is a polynomial $Q^\Phi_{\av} \in \Q[\Y]$ such that
 \[
\rk_T(\Phi^{\preceq(\sv)}(\av)|F)=Q^\Phi_{\av}(\sv)
\]
whenever $\min(\sv)$ is sufficiently large. Moreover:
\begin{enumerate}
\item The dominant terms of $Q^\Phi_{\av}$ only depend on the $T$-differential field extension $F\langle\Theta(\av)\rangle$.
\item The coefficient of $\Y^\dv$ in $Q^\Phi_{\av}$ times $d_1!\cdots d_k!$ is equal to the maximal size of a subset $B\subseteq F\langle\Theta(\av)\rangle$ such that the tuple $(\theta b)_{\theta \in \Theta,b \in B}$ is $\dcl$-independent over $F$.
\end{enumerate}
\end{corollary}


\end{document}